\documentclass[a4paper,10pt,reqno]{amsart}

\usepackage[english]{babel}
\usepackage[utf8]{inputenc}
\usepackage{graphicx}
\usepackage[colorlinks=true,linkcolor=black,urlcolor=black,citecolor=black]{hyperref}
\usepackage{hyperref}
\usepackage{amsfonts}
\usepackage{amssymb,amsmath}
\usepackage{amsthm}
\usepackage{fixmath}
\usepackage{enumerate}
\usepackage[shortlabels]{enumitem}
\usepackage{mathtools}

\newcommand{\x}{{{x}}}
\newcommand{\xx}{{{\x}^\prime}}
\newcommand{\dist}{{|\x-\xx|}}
\newcommand{\distn}{{|\x-\xx|^{n+2s}}}

\newcommand{\kernel}{\gamma}
\newcommand{\kernelh}{\hat{\gamma}}

\newcommand{\kerneld}{\widetilde{\kernel}_K}

\newcommand{\kerd}{\kernel(\x,\xx;\delta)}
\newcommand{\kerdd}{\kerneld(\x,\xx;\delta)}
\newcommand{\kermu}{\kernel(\x,\xx;\parmu)}

\newcommand{\kernelmax}{C_\gamma^{1}}

\newcommand{\deltapiv}{\delta^{\star}}
\newcommand{\deltamin}{\delta_{\min}}
\newcommand{\deltamax}{\delta_{\max}}
\newcommand{\smin}{s_{\min}}
\newcommand{\smax}{s_{\max}}
\newcommand{\Clow}{C_{1}}
\renewcommand{\Cup}{C_{2}}

\newcommand{\s}{{\hat{s}}}

\newcommand{\parmu}{{{\mu}}}
\newcommand{\heps}{{\hat{\varepsilon}}}
\newcommand{\R}{\mathbb{R}}
\renewcommand{\L}{\mathcal{L}}
\renewcommand{\SS}{\mathbb{S}^{n-1}}

\newcommand{\D}{\Omega}
\newcommand{\DI}{\Omega_{\delta}}
\newcommand{\DD}{{\D\cup\DI}}
\newcommand{\DL}{\widetilde{\Omega}}

\newcommand{\X}{X}
\newcommand{\V}{V}

\newcommand{\Vd}{{\V^{\prime}}}
\newcommand{\XD}{\X}
\newcommand*{\VN}[1]{{\V^{#1}_N}}

\renewcommand{\P}{\mathcal{P}} 
\newcommand{\Ps}{\P^s}

\newcommand*{\HD}[1]{{H^{#1}_\D(\R^n)}}

\newcommand{\VD}[1]{{\V^{#1}_\delta}}

\newcommand{\sconst}[1]{C_{#1}}
\newcommand{\sconstv}[1]{C_{#1}}

\renewcommand{\d}{\mathop{}\!\mathrm{d}}
\newcommand*{\norm}[1]{{\left\lVert#1\right\rVert}}
\newcommand*{\norms}[1]{{\left|{#1}\right|}}

\newcommand*{\dual}[1]{{\langle{#1}\rangle}}
\newcommand{\tr}[1]{{\widetilde{#1}}}
\newcommand{\spann}{\mathrm{span}}

\newcommand{\diam}{\mathrm{diam}}

\newcommand{\ad}{{\widetilde{a}_K}}
\newcommand{\adM}{{\widetilde{a}_{M}}}
\newcommand{\adr}{{\widetilde{a}_{M,\rho}}}
\renewcommand{\a}{{a}}

\numberwithin{equation}{section}
\newtheorem{theorem}{Theorem}[section]
\newtheorem{lemma}[theorem]{Lemma}
\newtheorem{proposition}[theorem]{Proposition}

\newtheorem{remark}{Remark}[section]
\newtheorem{corollary}{Corollary}[section]

\usepackage{amsopn}

\begin{document}
\title[Model order reduction for nonlocal models]{Affine approximation of parametrized kernels and model order reduction for nonlocal and fractional Laplace models}

\author{Olena Burkovska}
\address{
Computational and Applied Mathematics, Oak Ridge National Laboratory,
Oak Ridge, TN 37831, USA; Department of Scientific Computing, Florida State University, 400 
Dirac Science Library, Tallahassee, FL 32306-4120, USA}
\email{burkovskao@ornl.gov, oburkovska@fsu.edu}

\author{Max Gunzburger}
\address{Department of Scientific Computing, Florida State University, 400 
Dirac Science Library, Tallahassee, FL 32306-4120, USA}
\email{mgunzburg@fsu.edu}

\begin{abstract}
We consider parametrized problems driven by spatially nonlocal integral operators with parameter-dependent kernels.
In particular, kernels with varying nonlocal interaction radius $\delta > 0$ and fractional Laplace kernels, parametrized by the fractional power $s\in(0,1)$, are studied.
In order to provide an efficient and reliable approximation of the solution for different values of the parameters, we develop the reduced basis method as a parametric model order reduction approach. 
Major difficulties arise since the kernels are not affine in the parameters, singular, and discontinuous. Moreover, the spatial regularity of the solutions depends on the varying fractional power $s$.
To address this, we derive regularity and differentiability results with respect to $\delta$ and $s$, which are of independent interest for other applications such as optimization and parameter identification. 
We then use these results to construct affine approximations of the kernels by local polynomials. 
Finally, we certify the method by providing reliable a posteriori error estimators, which account for all approximation errors, and support the theoretical findings by numerical experiments.
\end{abstract}



\maketitle

\section{Introduction}

Nonlocal models are a broad category of mathematical models which arise in many areas of science, such as, e.g., solid and fluid mechanics, contact mechanics, subsurface flows, turbulence modeling, image analysis~\cite{gilboa2007}, and finance~\cite{cont2004}. In particular, nonlocal diffusion operators are used to model anomalous diffusion processes and an important example is given by the integral fractional Laplacian. 

While nonlocal models better reflect many physical processes, in many practical applications the precise model parameters are unknown.
In this case, one requires not only to approximate the solution as a function of the spatial variable, but also as a function of the model parameters.
Typically, such evaluations are performed for many instances of the parameter and model order reduction techniques, such as reduced basis methods (RBM), become attractive tools for reducing the computational complexity; see, e.g.,~\cite{hesthaven2015,quarteroni_rbm} and the references therein. 

In this paper we propose and analyze computationally efficient and reliable approximations, based on the reduced basis method, of the parametrized problem involving the nonlocal operator given as
\begin{equation}
-\L(\parmu) 
u(\x):=2\int_{\R^n}(u(\x)-u(\xx))\kermu\d\xx,\quad\x\in\D,
\label{nonl_operator0}
\end{equation}
where the kernel $\kermu$ is a non-negative symmetric function and $\D\subset\R^n$ is a bounded domain. The parameter vector $\parmu\in \P \subset \R^p$ collects the kernel parameters and determines the qualitative nature of the associated problem 
\begin{subequations}
\begin{align}
-\L(\parmu) u(\x;\parmu) &= f(\x;\parmu) &&\text{for } \x \in \D, \\
u(\x,\parmu) &= 0 &&\text{for } \x \in \R^n \setminus \D.
\label{nonl_volumecon0}
\end{align}\label{nonl_problem0}\end{subequations}
Here, $f$ is a data term, which may also depend on the parameters, and the problem is endowed with homogeneous volume constraints on the complement of the domain.

We focus on two cases of particular interest: First, we consider a general class of kernels $\kernel\colon \R^n\times\R^n \to \R$, where the nonlocal operator is given concretely by
\begin{equation}
-\L(\delta) u(\x) = 2\int_{B_\delta(\x)}(u(\x)-u(\xx))\kernel(\x,\xx)\d\xx,
\label{nonl_operator_delta0}
\end{equation}
with $B_{\delta}(\x)$ denoting the Euclidean ball of a radius $\delta>0$ centered at $\x$.
The corresponding problem is parametrized by $\parmu = \delta \in [\deltamin,\deltamax]$ for \(0 < \deltamin < \deltamax < \infty\), which describes the extent of the nonlocal interactions. Here, the parametrized kernel $\kernel(\x,\xx;\delta)$ arises from a truncation of the kernel $\gamma(\x,\xx)$ to the strip where $\dist < \delta$. Problems of this form are analyzed in~\cite{Dunonlocal2012}, and arise, e.g., in the peridynamics model~\cite{silling2000}.
Second, as a particular choice of kernels we also consider (truncated) fractional Laplace type kernels, with the fractional power $s\in(0,1)$ being the parameter in the model, i.e., for $\parmu=s \in [s_{\min},s_{\max}] \subset (0,1)$, we consider kernels of the following form
\begin{equation}
\kernel(\x,\xx;s) = 
\begin{cases}
\displaystyle \frac{c_{n,s}}{2\norms{\x-\xx}^{n+2s}} &\text{for } \dist < \delta, \\
0 &\text{else},
\end{cases}
\label{nonl_kernel_s0}
\end{equation}
where $c_{n,s} = (2^{2s}s\Gamma(s+n/2))/(\pi^{n/2}\Gamma(1-s))$.
In this case, we consider $\delta \in (0,\infty]$ to be a given and fixed quantity. We note that in the case $\delta = \infty$, the problem~\eqref{nonl_problem0} turns into the fractional Laplace equation in integral form with homogeneous volume constraints
\begin{subequations}
\begin{align}
(-\upDelta)^s u(\x;s) &= f(\x) &&\text{for } \x \in \D, \\
u(\x,s) &= 0 &&\text{for } \x \in \R^n \setminus \D,
\end{align}\label{fractional_problem0}\end{subequations}
which is a model of particular relevance; see, e.g., \cite{DELIAlaplacian,acosta2017}.

We point out that the reduced basis method is not a new approach, and its beginning is traced back to structural engineering applications; see, e.g.,~\cite{ALMROTH,noor1980reduced}. 
By now, a strong mathematical theory for the method has been developed and it has been significantly expanded to various applications.
However, the RBM has been mainly developed for local problems, i.e., problems governed by parametrized partial differential equations (PDEs). 
Despite the fact that nonlocal problems benefit more from model order reduction -- due to the reduced sparsity of the underlying discrete system and, as a consequence, the high computational cost -- the potential of the method for nonlocal problems is still not fully explored. However, we refer to recent works on the reduced basis method and proper orthogonal decomposition for some nonlocal problems~\cite{WITMANpod2017,GUANrbm2017,ANTILrb2018}.

The distinctive property of the RBM in contrast to many other approximation techniques lies in the fact that it does not try to perform an approximation of the underlying solution space, but rather an approximation of the parametric manifold. Typically, by means of a greedy search algorithm, the method captures information about parametric variations of the model. 
The efficiency of the method is gained by the so-called offline-online decomposition of the computational routine, and to do so, it is crucial that the parametric quantities of the model are affine in the parameters. 
In brief, the offline-online procedure splits the parameter-independent computations (offline) from the parameter dependent ones (online). The first phase is computationally expensive, but performed only once, while the latter one is computationally cheap and can be executed multiple times for different instances of an input parameter, which later allows to perform computations in a multi-query context.

In the situation of  non-affine parameter dependency, one has to approximate the bilinear and linear forms by corresponding affine counterparts.
To do so, one typically resorts to empirical interpolation (EIM),~\cite{Barrault} or discrete empirical interpolation (DEIM)~\cite{deim} methods. 
We note that for the problems under consideration, the dependency on the parameters is not affine, which requires special attention. In the case of $s$, the nature of the singularity changes with the parameter, which leads to different regularity properties of the solutions.
Moreover, the integral kernel is discontinuous at $\dist = \delta$ and has a singularity for $\dist \to 0$. This makes it impossible to apply empirical interpolation in a straightforward fashion, since it is designed for continuous and bounded functions. 
While the continuity condition can be relaxed by using a generalized empirical interpolation method (GEIM)~\cite{Maday2013}, in the present setting the choice of interpolating functional is not obvious.

To circumvent this difficulty, we develop an affine approximation of the bilinear form based on interpolation with (local) polynomials. 
Here, we take into account the specific regularity of the bilinear form with respect to $\delta$ and $s$, which results in an efficient and offline-online separable method. 
In the case of $s$, we prove an exponentially convergent approximation based on the regularity result for the bilinear form. A byproduct of our analysis is the parametric regularity of the solution in $\delta$ for a general class of kernels and in $s$ for the fractional Laplace problem, which is of independent interest.  In particular, we show the Lipschitz continuous differentiability of the solution with respect to $\delta$ and $C^\infty$-regularity with respect to $s$. To derive these results we strongly rely on the spatial regularity of the solution~\cite{grubb2015,burkgunz2018}. 
To the best of our knowledge, these are the first results on the study of the smoothness of the solution with respect to the nonlocal interaction radius and power of the integral fractional Laplacian.

We comment on the existing works on nonlocal model order reduction, in comparison to the approach we follow. In~\cite{GUANrbm2017} the reduced basis method is applied in the context of uncertainty quantification for nonlocal problems with random, but affine, coefficients. 
The reduced basis method for the power of the spectral fractional Laplacian has been recently studied in~\cite{ANTILrb2018}.  Although the spectral fractional Laplace problem, formulated via an extension formulation~\cite{caffarelli2007}, is different from the integral form studied here, they face a similar problem due to the non-affinity and singularity of the parametric functions. To treat it a ``cut-off'' procedure of the computational domain is employed, however the resulting approach is tied to the underlying finite element discretization. In the following, we present an affine approximations, which can be performed directly in the continuous formulation, and account for all incurred approximation errors in the derived a posteriori error estimates. 

Finally, we briefly comment on possible direct extensions of our work: By combining the developed approaches, we can directly extend the method to the case of a truncated Laplace kernel with both $\delta$ and $s$ as parameters. Moreover, additional dependencies on coefficients of the form
\[
\widehat{\kernel}(\x,\xx;\delta,s,\widehat{\mu}) = 
\sigma(\x,\xx;\widehat{\mu}) \, \kernel(\x,\xx;\delta,s),
\]
where \(0 < \sigma_{\min} \leq \widehat{\sigma} \leq \sigma_{\max} < \infty\), can easily be incorporated in to the resulting approach, as long as \(\sigma(\x,\xx;\widehat{\mu})\) is affine separable in \(\widehat{\mu}\); cf.~\cite{GUANrbm2017}.

The rest of the paper is structured as follows. In Section~\ref{sec:preliminaries}, we introduce the necessary function spaces and recall some preliminary results, which will be used in the rest of the paper. In Section~\ref{sec:regularity_delta} we analyze the parametric regularity with respect to $\delta$ and in Section~\ref{sec:affine_delta} we utilize the obtained results to construct an affine approximation of the problem with respect to $\delta$. Section~\ref{sec:regularity_s} studies a parametric smoothness of the bilinear form and the solution for the fractional Laplace problem.
Then, in Section~\ref{sec:affine_s}, these findings are used to for an affine approximation of the bilinear form based on Chebyshev interpolation. Eventually, in Section~\ref{sec:rbm}, we piece together the reduced basis approximation for both parameters. Finally, numerical results that illustrate our theoretical findings are given in Section~\ref{sec:numerics}.


\section{Preliminaries}\label{sec:preliminaries}
Let $\D\subset\mathbb{R}^n$ be a bounded domain with a Lipschitz boundary. We denote by $L^p(\D)$, $p \in [1,\infty]$ the usual Lebesgue spaces.
\subsection{General truncated kernels}
For \(\delta > 0\), we denote by $B_{\delta}(\x)$ {the ball} of radius $\delta$ centered at $\x$, $B_{\delta}(\x):=\{\xx\in\R^n\colon\dist\leq\delta\}$.
Then, for  $0<\deltamin<\deltamax<\infty$, we introduce a truncated kernel $ \kerd\colon\R^n\times\R^n\times[\deltamin,+\infty]\to\R$, which is a non-negative symmetric (w.r.t.\ $\x$ and $\xx$) function, and for all $\x\in\R^n$ it fulfills the following conditions: 
\begin{equation}
 \begin{cases}
   \kerd\geq 0 &\forall\xx\in B_{\delta}(\x),\\
    \kerd=0 &\forall\xx\in\R^n\setminus B_{\delta}(\x),\\
   \kerd\geq\gamma_0>0 &\forall\xx\in B_{\deltamin/2}(\x).
  \end{cases}\label{kernel_prop1}\tag{H1}
\end{equation}
Note that we allow the truncation parameter to be infinite, in order to include the fractional Laplace problem into the analysis. If $\delta=+\infty$, by a slight abuse of notation, we simply write $\kernel(\x,\xx)$, implying $\kernel(\x,\xx):=\kernel(\x,\xx;+\infty)$.

In addition, we assume that there exists a function $\hat{\kernel}{\colon}\R^+\to\R^+$, such that
\begin{equation}
\kerd\leq\kernelh(|\x-\xx|), \quad\mbox{and }\; |\xi|^{n-1}\kernelh(|\xi|)\in L^1((\deltamin,\deltamax)),\label{kernel_prop2} \tag{H2}
\end{equation}
and we denote 
\begin{equation}
\kernelmax:=\omega_{n-1}\int_{\deltamin}^{\deltamax}|\xi|^{n-1}\hat{\kernel}(|\xi|)\d\xi=\int_{B_{\deltamax}(0)\setminus B_{\deltamin}(0)}\hat{\kernel}(|z|)\d z,
\end{equation}
where $\omega_{n-1}$ is the surface measure of the $n-1$ dimensional unit sphere $\SS:=\{\x\in\R^n\colon \norms{\x}=1\}$ embedded in dimension $n$, given explicitly as 
\begin{equation}
\omega_{n-1}:=\frac{2\pi^{n/2}}{\Gamma(n/2)}.
\label{volume_unit_sphere}
\end{equation}
In certain cases, we need to require a stronger assumption on $\kernel$, namely that $\hat{\kernel}(|\x-\xx|)$ in~\eqref{kernel_prop2} is such that $|\xi|^{n-1}\hat{\kernel}(|\xi|)\in L^\infty((\deltamin,\deltamax))$, i.e., there exist $C_{\kernel}>0$, such that
\begin{equation}
\norms{\hat{\kernel}(|\xi|)|\xi|^{n-1}}\leq C_{\kernel}.\label{kernel_prop3} \tag{H3}
\end{equation}

For $\delta\in[\deltamin,\deltamax]$, we define the interaction domain $\DI\subset\mathbb{R}^d\setminus\D$ corresponding to $\D$ as follows $\DI:=\{\xx\in\mathbb{R}^d\setminus\D\colon \norms{\x-\xx}<\delta,\; \x\in\D \}$. 
In terms of these notations, we define the nonlocal operator
\begin{equation}
-\L(\delta) u(\x):=2\int_{\DD}(u(\x)-u(\xx))\kerd\d\xx,\quad\x\in\D.
\label{nonl_operator}
\end{equation}
Then, for a given function $f$, e.g., $f\in L^2(\D)$, we consider the problem~\eqref{nonl_problem0}.
However, due to the finite range of interaction the volume constraint~\eqref{nonl_volumecon0} needs to be enforced only on $\DI$ instead of the whole $\R^n\setminus\D$. It is the analogue to the Dirichlet type boundary condition imposed on $\partial\D$ for the local case and is essential for the well-posedness of the nonlocal problem. 

\subsection{Fractional Laplace type kernels}
A special focus in this paper will lie on (truncated) fractional Laplace type kernels. That is, the kernels parametrized by the fractional power $s\in(0,1)$ of the following form:
\begin{equation}
 \kernel(\x,\xx;s)=\begin{dcases}
 \frac{1}{\distn}, &{\xx\in B_{\delta}(\x)},\\
 0, &\text{otherwise}.
 \end{dcases}\label{kernel:FL_delta}
\end{equation}
Here, $\delta$ is considered to be a given fixed parameter.
We note that for $\delta=+\infty$,~\eqref{kernel:FL_delta} becomes the classical fractional Laplace kernel
and $-\L$ reduces to the fractional Laplace operator $(-\upDelta)^s$ up to the scaling constant $c_{n,s}/2$:
\begin{equation}
(-\upDelta)^s u(\x):=c_{n,s}\int_{\R^n}\frac{u(\x)-u(\xx)}{\distn}\d\xx, \quad c_{n,s}=\frac{2^{2s}s\Gamma(s+\frac{n}{2})}{\pi^{n/2}\Gamma(1-s)}.
\end{equation}
In~\cite{DELIAlaplacian}, the convergence of the nonlocal solution of~\eqref{nonl_problem0} with the (truncated) kernel~\eqref{kernel:FL_delta} to the solution of~\eqref{nonl_problem0} with the fractional Laplace kernel is analyzed.

In order to obtain the usual fractional Laplace problem from~\eqref{nonl_problem0}, we incorporate the scaling factor into the right hand side, by defining
\begin{equation}
\label{eq:RHS_s}
f(\x; s) := \frac{2F(\x)}{c_{n,s}} \quad\text{for } \x \in \D,
\end{equation}
for some given $F$, e.g., $F \in L^2(\Omega)$. Then the solution $u$ of~\eqref{nonl_problem0} corresponds to the solution of the problem $(-\upDelta)^s u = F$ for $\delta = \infty$.

\subsection{Function spaces}
By $H^1(\D)$ and $H^1_0(\D)$ we denote the usual Sobolev spaces,
and introduce the constrained $L^2$-space, $L^2_{\D}(\DD):= \{v\in  L^2(\DD)\colon v =
0\;\; \text{a.e. on}\;\; \DI\}$, which is isometrically isomorphic to $L^2(\Omega)$. For $u,v\in L_\D^2(\DD)$ we define the associated inner product
\begin{equation}
\label{energy_delta}
((u,v))_{\delta}:=\int_{S_\delta}\left(u(\x)-u(\xx)\right)\left(v(\x)-v(\xx)\right)\kernel(\x,\xx)\d(\xx,\x),
\end{equation}
where $S_\delta$ is the strip
\begin{equation}
S_{\delta}=\{(\x,\xx)\in\R^{2n}\colon 
\dist\leq\delta\}.
\end{equation}
The corresponding norm is given by $\norm{u}_{\delta}^2 = ((u,u))_{\delta}$.
We introduce the following energy and constrained energy spaces
\[
\XD_\delta:=\{v\in L^2(\DD)\colon\; \norm{v}_{\delta}<\infty\},\quad \V_\delta:=\{v\in \XD_\delta\colon\;  
v=0 \;\;\text{ a.e. on }\DI\},
\]
which are Hilbert spaces, equipped with the inner products
$(u,v)_{\V_\delta}:=((u,v))_{\delta}$ and 
$(u,v)_{\XD_\delta}:=(u,v)_{L^2(\DD)} + ((u,v))_{\delta}$.
We denote by $\V^{\prime}_{\delta}$ the dual space of $\V_{\delta}$, and by $\dual{\cdot,\cdot}$ the extended $L^2(\DD)$ duality pairing between these spaces. 
\begin{remark}
Often it will be convenient to consider a common spatial domain, independent of $\delta$, for functions $u\in\V_{\delta}$. Thus, we consider an extension by zero of $u$ outside of $\D$, which, by a slight abuse of notation also denoted by $u$, and as a common spatial domain we chose the whole $\R^n$. It is clear that these functions are equivalent and we will use these equivalence representations throughout the paper.
\end{remark}
We assume that the kernel $\kernel$ is such that following nonlocal Poincar\'e inequality holds:
\begin{equation}
 \norm{v}_{L^2(\DD)}\leq C_P\norm{v}_{\V_{\delta}}\qquad\forall 
v\in\V_{\delta},\label{poincare}
\end{equation}
where $C_{P}>0$ is a Poincar\'e constant, independent of $\delta$. In particular, the condition~\eqref{poincare} is satisfied for the fractional Laplace type kernels, see also~\cite{Dunonlocal2012}, where different classes of kernels are discussed for which this property also holds.

Throughout the paper we often make use of the fractional Sobolev spaces, which are defined as follows: Let $\DL$ be given either by $\D$ or $\R^n$. For $s\in(0,1)$ we define
\begin{equation*}
 H^s(\DL):=\big\{v\in L^2(\DL)\colon |v|_{H^s(\DL)}<\infty\big\}
\end{equation*}
with Gagliardo seminorm
\begin{equation*} 
|v|_{H^s(\DL)}^2:=\int_{\DL}\int_{\DL}\frac{|v(\x)-v(\xx)|^2}{\distn}
\d\xx\d\x.
\end{equation*}
For $s>1$ not an integer, we define $H^s(\DL)$, $s=m+\sigma$ with $m\in\mathbb{N}$ and $\sigma \in (0,1)$, as
\begin{equation*}
 H^s(\DL):=\{v\in H^m(\DL)\colon {D^\alpha v} \in {H^\sigma(\DL)} \text{ for } 
|\alpha|=m\},
\end{equation*}
together with the semi-norm
\begin{equation*} 
\norms{v}^2_{H^s(\DL)}=\norms{v}^2_{H^m(\DL)}+\sum_{|\alpha|=m}\norms{D^\alpha v}^2_{
H^\sigma(\DL) } .
\end{equation*}
The space $H^s(\DL)$ is a Hilbert space that is endowed with the norm 
$ \norm{v}^2_{H^s(\DL)}=\norm{v}^2_{L^2(\DL)}+\norms{v}^2_{H^s(\DL)}$.
Additionally, we define the space incorporating the volume constraints given by
\begin{equation*}
 H_{\D}^s(\R^n):=\{v\in H^s(\R^n)\colon v=0 \text{ on }\ \R^n\setminus\D\},
\end{equation*}
that is endowed with the semi-norm of $H^s(\R^n)$. For negative exponents, we define the associated spaces by duality $H^{-s}(\Omega) = (H_{\D}^s(\R^n))^{\prime}$.

We note that for the case of a (truncated) fractional Laplace kernel, to highlight the inclusion of the parameter $s$, we use the notation $\VD{s}$ instead of $\V_{\delta}$. 
Moreover, the nonlocal space $\VD{s}$ is equivalent to $\HD{s}$, which implies that we can equivalently work with either $\VD{s}$ or $\HD{s}$. In particular, for $v\in\VD{s}$, $s\in(0,1)$, $C>0$, we have
\begin{equation}
{C}\norm{v}_{\HD{s}}\leq\norm{v}_{\VD{s}}\leq\norm{v}_{\HD{s}}.\label{equiv_1}
\end{equation}
Also, for any $s_1\leq s_2$, there exists $\sconst{s_1,s_2}>0$, such that for all $v\in\HD{s_2}$ it holds
\begin{align}
\norm{v}_{\HD{s_1}}\leq{\sconst{s_1,s_2}}\norm{v}_{\HD{s_2}},\quad\norm{v}_{\VD{s_1}}\leq{\sconstv{s_1,s_2}}\norm{v}_{\VD{s_2}}.
\label{equiv_2}
\end{align}

\subsection{Weak formulation}
First, we provide a result that states {the} equivalence of nonlocal spaces $\V_{\delta}$ w.r.t.\ $\delta\in[\deltamin,\deltamax]$.
\begin{proposition}[Equivalence of nonlocal spaces]
Let $\kernel$ satisfy~\eqref{kernel_prop1} and \eqref{kernel_prop2}. Then, for all $\delta\in[\deltamin,\deltamax]$, the spaces $\V_{\delta}$ are all equivalent, and for some $\deltapiv\in[\deltamin,\deltamax]$, we have the following norm bound
\begin{equation}
\Clow\norm{v}_{\V_{\deltapiv}}\leq\norm{v}_{\V_{\delta}}\leq\Cup\norm{v}_{\V_{\deltapiv}},\quad\forall v\in\V_{\delta},\label{equiv_nonloc_spaces}
\end{equation}
where $1/\Clow:=\sqrt{1+4C_{P}^2\kernelmax}$, $\Cup=1$ if $\deltapiv>\delta$, and $\Clow:=1$, $\Cup=\sqrt{1+4C_{P}^2\kernelmax}$ else.
\end{proposition}
\begin{proof}
Let $\deltapiv>\delta$, and for $u\in L^2_\D(\R^n)$ we consider
\begin{align*}
\norm{u}_\delta^2&=\int_{S_{\delta}}(u(\xx)-u(\x))^2\kernel(\x,\xx;\delta)\d(\xx,\x)\\
&=\norm{u}^2_{{\deltapiv}}-\int_{S_{\deltapiv}\setminus S_{\delta}}(u(\xx)-u(\x))^2\kernel(\x,\xx;\deltapiv)\d(\xx,\x).
\end{align*}
It is clear that, if $u\in\V_{\deltapiv}$, then  $u\in\V_{\delta}$, and $\norm{u}_{\V_{\delta}}\leq\norm{u}_{\V_{\deltapiv}}$. On the other hand, if $u\in\V_{\delta}$, then 
\begin{align*}
\norm{u}^2_{\deltapiv}&\leq\norm{u}_{\V_{\delta}}^2+\int_{S_{\deltamax}\setminus S_{\deltamin}}(u(\xx)-u(\x))^2\hat{\kernel}(|\x-\xx|)\d(\xx,\x)\\
&\leq \norm{u}_{\V_{\delta}}^2+4\kernelmax\norm{u}^2_{L^2(\D)}\leq (1+4C_P^2\kernelmax)\norm{u}^2_{\V_{\delta}},
\end{align*}
and, hence, $u\in\V_{\deltapiv}$. Applying the same arguments as above for the case $\deltapiv<\delta$ we obtain the corresponding result and conclude the proof.
\end{proof}
Since, the spaces $\{V_{\delta},\;\delta\in[\deltamin,\deltamax]\}$ are equivalent, for further study, it is convenient to have a common function space. For some reference $\deltapiv\in[\deltamin,\deltamax]$ we denote a pivot space $\V$, such that $\V\cong\V_{\delta}$, for all $\delta\in[\deltamin,\deltamax]$, and it is defined as
\begin{align*}
\V:=\{v\in L_\D^2(\D\cup\D_{\delta^\star})\colon\; \norm{v}_{\deltapiv}<\infty\},\;\;\mbox{and }\;\;\norm{v}_{\V}=\norm{v}_{\deltapiv}.
\end{align*}

For all $u,v\in\V$, we define the bilinear form $a:\V\times\V\to\R^n$, where we suppress the dependency on the parameters for now, as follows
\begin{align}
a(u,v):=
\int_{S_\delta} \left(u(\x)-u(\xx)\right)\left(v(\x)-v(\xx)\right)\kernel(\x,\xx)\d(\xx,x).
\label{bil_form}
\end{align}
Using the equivalence of nonlocal spaces w.r.t.\ $\delta$, we obtain that the bilinear form $a(\cdot,\cdot)$ is continuous and coercive on $\V\times\V$, i.e., for all $u,v\in\V$, we have
\begin{equation}\label{constants_a}
\a(u,v)\leq\gamma_{a}\norm{u}_{\V}\norm{v}_{\V},\quad
\a(u,u)\geq\alpha_{a}\norm{u}^2_{\V}, \quad \gamma_a:=\Cup^2, \quad\alpha_a:=\Clow^2.
\end{equation}
Then, by means of the nonlocal vector calculus~\cite{Dunonlocal2012}, we now pose the problem~\eqref{nonl_problem0} in the following weak form: For a given $f\in\V^{\prime}$,
find $u\in\V$, such that
\begin{equation}
a(u,v)=\dual{f,v}\quad\forall v\in\V.
 \label{nonl_linear_var}
\end{equation}
By the Lax-Milgram theorem, the problem~\eqref{nonl_linear_var} admits a unique 
solution, in addition, there exists a constant $C$, such that
\begin{equation}
\norm{u}_{\V}\leq C\norm{f}_{\V^{\prime}} \leq C_f,\label{bound_f}
\end{equation}
where we assume that $C_f$ is a constant independent of the parameters.

We recall a regularity result for the solution of~\eqref{nonl_linear_var} with the truncated fractional Laplace kernel~\eqref{kernel:FL_delta} stated in~\cite{burkgunz2018}, which is essentially based on the regularity results for the fractional Laplacian~\cite{grubb2015}.
\begin{theorem}\label{thm:regularity_linear}
Let $\D$ be a domain with $C^{\infty}$ boundary $\partial\D$, and let $f\in  H^r(\D)$, $r\geq -s$, and let $u\in\HD{s}$ be the solution of~\eqref{nonl_linear_var} with the kernel~\eqref{kernel:FL_delta} for $\delta>0$. Then, for any $\varepsilon>0$ there exists a constant $C>0$, such that
\begin{equation}
\norm{u}_{H_\D^{s+\alpha}(\R^n)}\leq C\norm{f}_{{H^r(\D)}},
\end{equation} \label{nonl_regularity}
where $\alpha=\min\{s+r,1/2-\varepsilon\}$. 
\end{theorem}


\section{Regularity with respect to \texorpdfstring{$\delta$}{Lg}}\label{sec:regularity_delta}
In this section we investigate the smoothness of the problem~\eqref{nonl_linear_var} with
respect to $\delta$. Here, we denote the parameter dependent bilinear form by $a(\cdot,\cdot;\delta)$, and assume that the data term is given by $f = F$, for some fixed $F \in L^2(\Omega)$. First, we show the Lipschitz continuity of the bilinear form and the solution, which is necessary for further application of the reduced basis method. In addition, under a regularity assumption, we prove also their differentiability w.r.t.\ $\delta$. 
\begin{proposition}
Let $\kernel$ satisfy~\eqref{kernel_prop1}--\eqref{kernel_prop3}.
Then the bilinear form $a(\cdot,\cdot;\delta)$ defined in~\eqref{bil_form} is Lipschitz continuous with respect to $\delta\in[\deltamin,\deltamax]$,  that is for all $u,v\in\V$, it holds 
with $L_a:=4C^2_{P}C_{\kernel}$ that
\begin{equation}
|a(u,v;\delta_1)-a(u,v;\delta_{2})|\leq L_{a}\norm{u}_{\V}\norm{v}_{\V}|\delta_1-\delta_2|, \quad\forall\delta_1,\delta_2\in[\deltamin,\deltamax].\label{eq:lipschitz_delta_a}
\end{equation}
\end{proposition}
\begin{proof}
For simplicity of notation we introduce the abbreviation: $u:=u(\x)$, $u^{\prime}:=u(\xx)$, $\kernel:=\kernel(\x,\xx;\deltamax)$. Then, for all $u,v\in\V$ we compute
\begin{multline*}
\norms{a(u,v;\delta_1)-a(u,v;\delta_2)}
=\norms{\int_{S_{\delta_2}\setminus S_{\delta_1}}(u^{\prime}-u)(v^{\prime}-v)\kernel}\\
\leq 4\norm{u}_{L^2(\D)}\norm{v}_{L^2(\D)}\norms{\int_{\delta_1}^{\delta_2}|\xi|^{n-1}\hat{\kernel}(|\xi|)\d\xi}
\leq 4C^2_{P}C_{\kernel}\norm{u}_{\V}\norm{v}_{\V}\norms{\delta_1-\delta_2},
\end{multline*}
which concludes the proof.
\end{proof}
\begin{proposition}[Lipschitz continutiy w.r.t. $\delta$]
The solution $u(\delta)\in\V$ of~\eqref{nonl_linear_var} is Lipschitz continuous with respect to $\delta\in[\deltamin,\deltamax]$, i.e., for all $\delta_1,\delta_2\in[\deltamin,\deltamax]$, it holds that
\begin{equation}
\norm{u(\delta_1)-u(\delta_2)}_{\V}\leq L_{u}|\delta_1-\delta_2|,\label{eq:lipschitz_delta_u}
\end{equation}
where $L_u:=L_aC_f/C_1$ and $C_1$ is defined in~\eqref{equiv_nonloc_spaces}.
\end{proposition}
\begin{proof}
For simplicity of notation we denote by $u_1:=u(\delta_1)$, and by $u_2:=u(\delta_2)$. Then, for $u_1,u_2,v\in\V$, using the Lipschitz continuity of $a(\cdot,\cdot;\delta)$, we can write
\begin{align*}
\norms{a(u_1-u_2,v;\delta_1)}&=\norms{a(u_2,v;\delta_2)-a(u_2,v;\delta_1)}\leq L_a\norm{u_2}_\V\norm{v}_{\V}|\delta_1-\delta_2|.
\end{align*}
Taking $v:=u_1-u_2\in\V$, and using~\eqref{bound_f} together with~\eqref{equiv_nonloc_spaces} we obtain
\begin{align*}
C_1\norm{u_1-u_2}^2_\V\leq\norm{u_1-u_2}^2_{\V_{\delta_1}}
\leq L_a C_f\norm{u_1-u_2}_\V|\delta_1-\delta_2|.
\end{align*}
\end{proof}

Next, we show that under appropriate conditions we can expect {}differentiability of $a(\cdot,\cdot;\delta)$ with respect to $\delta$. This result will be crucial for the derivation of {improved} a~posteriori error bounds {for} the reduced basis approximation. 
\begin{theorem}\label{thm:reg_delta}
Let $\kernel$ be radial, i.e., $\kernel(\x,\xx)=\kernelh(\norms{\x-\xx})$, and satisfies~\eqref{kernel_prop1}--\eqref{kernel_prop3}, then $a(\cdot,\cdot;\delta)$ is differentiable w.r.t.\ $\delta$, i.e., there exists a bounded bilinear form $a^{\prime}(\cdot,\cdot;\delta)$ such that $a^{\prime}(u,v;\delta):=\frac{\d}{\d\delta}a(u,v;\delta)$ for all $u,v\in\V$, $\delta\in[\deltamin,\deltamax]$. In particular, it holds
\begin{equation}
\label{eq:der_a}
a^{\prime}_\delta(u,v;\delta)
= \int_{\R^n} \int_{\partial B_\delta(\x)} \left(u(\x)-u(\xx)\right)\left(v(\x)-v(\xx)\right)\kernel(\x,\xx) \d\xx \d\x,
\end{equation}
where $\partial B_\delta(\x)$ is the surface of the ball of radius $\delta$ at $\x$, i.e.\ $\partial B_\delta(\x) = \{\,\xx \in \R^n \;:\; \dist = \delta \,\}$, and the inner integral is understood as a surface integral, together with the estimate
\begin{equation}
\norms{a^{\prime}_\delta(u,v;\delta)}\leq 
{C_{a^{\prime}}}\norm{u}_{L^2(\D)}\norm{v}_{L^2(\D)},
\label{eq:bdd_der_a}
\end{equation}
where $C_{a^{\prime}}=4\omega_{n-1}\delta^{n-1}\kernelh(\delta)$ and $\omega_{n-1}$ is defined in~\eqref{volume_unit_sphere}.
Moreover, if $u\in H^1_0(\D)$ and if $\kernelh$ is Lipschitz continuous with respect to $\delta$, i.e.,
\begin{equation}
\norms{\kernelh(\delta_1)-\kernelh(\delta_2)}\leq L_{\kernel}\norms{\delta_1-\delta_2},\quad L_\kernel>0,\quad \delta_1,\delta_2\in[\deltamin,\deltamax],\label{eq:lipschitz_kernel}
\end{equation}
then $a^{\prime}_{\delta}(\cdot,\cdot;\delta)$ is also Lipschitz continuous and the following holds
\begin{equation}
\norms{a^{\prime}_{\delta}(u,v;\delta_1)-a^{\prime}_{\delta}(u,v;\delta_2)}\leq L_{a^{\prime}}\norm{u}_{H_0^1(\D)}\norm{v}_{L^2(\D)}\norms{\delta_1-\delta_2},\label{eq:lipschitz_der_a}
\end{equation}
with $L_{a^{\prime}}:=2\omega_{n-1}\delta_2^{n-1}\left(2C_P((n-1)\delta_2^{-1}\kernelh(\delta_1) + L_{\kernel}) + \kernelh(\delta_2)\right)$ for $\delta_1<\delta_2$.
\end{theorem}
\begin{proof}
Shifting the inner integral to $h = \xx-\x$, using the radiality of the kernel, changing the order of integration, and changing to polar coordinates, we obtain
\begin{align*}
a(u,v;\delta)=&\int_{\R^n}\int_{B_{\delta}(\x)}(u(\x)-u(\xx))(v(\x)-v(\xx))\kernel(\x,\xx)\d\xx\d\x\\
=&\int_{B_{\delta}(0)} g(h)\d h
= \int_0^\delta \rho^{n-1} \int_{\SS} g(\rho\xi) \d \xi \d \rho,
\end{align*}
where $\SS = \partial B_\delta(0)$ is the $(n-1)$-sphere, and the inner integral is abbreviated as
\begin{equation*}
g(h):=\int_{\R^n}
\left(u(\x)-u(\x+h)\right)
\left(v(\x)-v(\x+h)\right)
\kernelh\left(\norms{h}\right)\d \x.
\end{equation*}
Thus, for the derivative we clearly obtain that
\begin{multline}
\a^{\prime}_{\delta}(u,v;\delta):=\frac{\d}{\d\delta}\left(\a(u,v;\delta)\right)=\delta^{n-1}\int_{\SS}g(\delta\xi)\d\xi\\
=\delta^{n-1}\kernelh(\delta)\int_{\SS}\int_{\R^n}\left(u(\x)-u(\x+\delta\xi)\right)\left(v(x)-v(x+\delta\xi)\right)\d\x\d\xi.
\end{multline}
Changing the order of integration again, and substituting \(\xx = \x+\delta\xi\), we obtain the desired representation~\eqref{eq:der_a}.
To obtain the boundedness, we split the integral in four parts, which leads to
\begin{multline*}
\a^{\prime}_{\delta}(u,v;\delta)\\
=\delta^{n-1}\kernelh(\delta)\left(2\omega_{n-1}(u,v)_{L^2(\D)}-\int_{\SS}\int_{\R^n} [u(\x+\delta\xi)v(\x)+u(\x)v(\x+\delta\xi)]\d\x\d \xi\right)\\
=\delta^{n-1}\kernelh(\delta)\left(2\omega_{n-1}(u,v)_{L^2(\D)}-\int_{\SS}\int_{\R^n} [u(\x+\delta\xi) + u(\x-\delta\xi)] v(\x)\d\x\d \xi\right).
\end{multline*}
In the final line, we split the last term into two, shift one integral by $\delta\xi$, and recombine the result.
Applying the Cauchy-Schwarz inequality to the last term, we obtain the stated inequality~\eqref{eq:bdd_der_a}.

Next, we show~\eqref{eq:lipschitz_der_a}. For $\delta_1, \delta_2\in [\deltamin,\deltamax]$ with $\delta_1<\delta_2$ we compute
\begin{multline*}
\norms{a^{\prime}_{\delta}(u,v;\delta_1)-a^{\prime}_{\delta}(u,v;\delta_2)}
=
 \left|2\omega_{n-1}(\delta_1^{n-1}\kernelh(\delta_1)-\delta_2^{n-1}\kernelh(\delta_2))(u,v)_{L^2(\D)}\right.\\
\left.+\int_{\SS}\int_{\R^n}\left(\delta_2^{n-1}\kernelh(\delta_2)u(\x+\delta_2\xi)-\delta_1^{n-1}\kernelh(\delta_1)u(\x+\delta_1\xi)\right)v(\x)\d\x\d\xi \right.\\
\left.+\int_{\SS}\int_{\R^n}\left(\delta_2^{n-1}\kernelh(\delta_2)u(\x-\delta_2\xi)-\delta_1^{n-1}\kernelh(\delta_1)u(\x-\delta_1\xi)\right)v(\x)\d\x\d\xi \right|\\
=\left|2\omega_{n-1}(\delta_1^{n-1}\kernelh(\delta_1)-\delta_2^{n-1}\kernelh(\delta_2))(u,v)_{L^2(\D)}\right.\\
\left.+\delta_2^{n-1}\kernelh(\delta_2)\int_{\SS}\int_{\R^n}\left(u(\x+\delta_2\xi)-u(\x+\delta_1\xi)\right)v(\x)\d\x\d\xi\right.\\
\left.+\left(\delta_2^{n-1}\kernelh(\delta_2)-\delta_1^{n-1}\kernelh(\delta_1)\right)\int_{\SS}\int_{\R^n}u(\x+\delta_1\xi)v(\x)\d\x\d\xi  \right.\\
\left.+\delta_2^{n-1}\kernelh(\delta_2)\int_{\SS}\int_{\R^n}\left(u(\x-\delta_2\xi)-u(\x-\delta_1\xi)\right)v(\x)\d\x\d\xi\right.\\
\left.+\left(\delta_2^{n-1}\kernelh(\delta_2)-\delta_1^{n-1}\kernelh(\delta_1)\right)\int_{\SS}\int_{\R^n}u(\x-\delta_1\xi)v(\x)\d\x\d\xi  \right|.
\end{multline*}
Using the regularity $u\in H^1_0(\D)$, which also implies that $u \in H^1(\R^n)$, and applying the fundamental theorem of calculus, we can express
\begin{equation}
u(\x-\delta_2\xi)=u(\x-\delta_1\xi)+\int_{\delta_1}^{\delta_2}\nabla u(\x-\delta\xi)(-\xi)\d\delta,\label{eq:thm:derivative_u}
\end{equation}
and respectively for $u(\x+\delta_2\xi)$. Using Lipschitz continuity of $\kernelh$~\eqref{eq:lipschitz_kernel}, we can estimate
\begin{multline}
\norms{\delta_1^{n-1}\kernelh(\delta_1)-\delta_2^{n-1}\kernelh(\delta_2)}\leq\norms{\delta_1^{n-1}-\delta_2^{n-1}}\kernelh(\delta_1)+\delta_2^{n-1}\norms{\kernelh(\delta_1)-\kernelh(\delta_2)}\\
\leq\norms{\delta_1-\delta_2}\kernelh(\delta_1)\sum_{j=0}^{n-2}\delta_1^{n-j-2}\delta_2^j+\delta_2^{n-1}L_{\kernel}\norms{\delta_1-\delta_2}
\leq\norms{\delta_1-\delta_2}((n-1)\delta_2^{n-2}\kernelh(\delta_1)+L_{\kernel}\delta_2^{n-1}).
\label{eq:thm:deltas}
\end{multline}
Then, combining~\eqref{eq:thm:derivative_u},~\eqref{eq:thm:deltas} and using Cauchy-Schwarz inequality, we can estimate
\begin{multline*}
\norms{a^{\prime}_{\delta}(u,v;\delta_1)-a^{\prime}_{\delta}(u,v;\delta_2)}
\leq4\omega_{n-1}\norms{\delta_1^{n-1}\kernelh(\delta_1)-\delta_2^{n-1}\kernelh(\delta_2)}\norm{u}_{L^2(\D)}\norm{v}_{L^2(\D)}\\
+2\delta_2^{n-1}\kernelh(\delta_2)\omega_{n-1}\norms{\delta_1-\delta_2}\norm{\nabla u}_{L^2(\D)}\norm{v}_{L^2(\D)}\\
\leq 2\omega_{n-1}\delta_2^{n-1}\left(2C_P((n-1)\delta_2^{-1}\kernelh(\delta_1) + L_{\kernel}) + \kernelh(\delta_2)\right)\norm{u}_{H_0^1(\D)}\norm{v}_{L^2(\D)}\norms{\delta_1-\delta_2}.
\end{multline*}
This concludes the proof.
\end{proof}

\begin{corollary}[$\delta$-regularity of the solution]
Let $u(\delta)\in\V$, $\delta\in[\deltamin,\deltamax]$, be the solution of~\eqref{nonl_linear_var}. Then, under the conditions of Theorem~\ref{thm:reg_delta} (namely, $\gamma$ is radial and $\kernelh$ is Lipschitz, and $u(\delta)\in H^1_0(\Omega)$ for all $\delta\in[\deltamin,\deltamax]$ with a uniform bound), it holds that $u$ is differentiable w.r.t. $\delta$, i.e., there exist $u^{\prime}_{\delta}\in\V$, $u^{\prime}_{\delta}(\delta):=\frac{\d}{\d\delta}u(\delta)$, and, moreover, $u^{\prime}_{\delta}$ is Lipschitz continuous w.r.t. $\delta$, i.e., $u\in C^{1,1}([\deltamin,\deltamax],\V)$.
\end{corollary}
\begin{proof}
It is easy to show that $u^{\prime}_{\delta}\in\V$ is the solution of the sensitivity equation
\begin{equation*}
a(u^{\prime}_\delta(\delta),v;\delta)=-a_{\delta}^{\prime}(u(\delta),v;\delta),\quad\forall v\in\V.
\end{equation*}
In fact, this result can be directly derived by subtracting the problems for $\delta$ and $\delta+\tau$, $\tau>0$, forming the difference quotient, and passing to the limit for $\tau\to 0$. Next, we prove the Lipschitz continuity of $u^\prime_{\delta}$. For $\delta,\widetilde{\delta}\in[\deltamin,\deltamax]$ we consider
\begin{multline*}
a(u^{\prime}_{\delta}(\delta)-u^{\prime}_{\delta}(\widetilde{\delta}),v;\delta)= a(u^{\prime}_{\delta}(\widetilde{\delta}),v;\widetilde{\delta})-a(u^{\prime}_{\delta}(\widetilde{\delta}),v;\delta)
+ a^{\prime}_{\delta}(u(\widetilde{\delta}) - u(\delta),v;\delta)\\
+ a^{\prime}_{\delta}(u(\widetilde{\delta}),v;\widetilde{\delta})-a^{\prime}_{\delta}(u(\widetilde{\delta)},v;\delta)\\
\leq\left(L_a\lVert{u^{\prime}_{\delta}(\widetilde{\delta})}\rVert_{\V}
+{C_{a^{\prime}}}C_P^2L_{u}
+L_{a^{\prime}}C_P\lVert{u(\widetilde{\delta})}\rVert_{H^1_0(\D)}\right)
\norm{v}_{\V}|\delta-\widetilde{\delta}|,
\end{multline*}
where the last estimate has been obtained by applying~\eqref{eq:lipschitz_delta_a},~\eqref{eq:lipschitz_delta_u},~\eqref{eq:bdd_der_a},~\eqref{eq:lipschitz_der_a}, and~\eqref{poincare}. Then, taking $v=u^{\prime}_{\delta}(\delta)-u^{\prime}_{\delta}(\widetilde{\delta})$ and using the coercivity of $a$ from~\eqref{constants_a}, we obtain the desired result. 
\end{proof}

\begin{remark}
We note that for the truncated fractional Laplace kernel~\eqref{kernel:FL_delta}, the condition $u\in H^1_0(\D)$ holds true if $s>1/2$, and $f \in L^2(\D)$; see Theorem~\ref{thm:regularity_linear}.
\end{remark}


\section{Affine approximation with respect to \texorpdfstring{$\delta$}{Lg}}\label{sec:affine_delta}
Utilizing the regularity results derived in the previous section, we present suitable approximation techniques that resolve the non-affine structure of the problem.
\subsection{Approximation of the kernel with respect to \texorpdfstring{$\delta$}{Lg}}

For $K\in\mathbb{N}$, we consider the following partitioning of the interval $[\deltamin,\deltamax]$:
$$0<\deltamin:=\delta_0<\delta_1<\cdots<\delta_{K}:=\deltamax<\infty,$$
with the step size $\Delta\delta:=\max\{\delta_k-\delta_{k-1}\}$, $k=1,\dots,K$. Then we approximate $\kerd$ by $\kerdd$, defined as follows
\begin{align}
\kerdd:=\sum_{k=0}^K\Theta_k^{\delta}(\delta)\kernel(\x,\xx;\delta_k).
\label{kernel_approx_gen}
\end{align}
In the following, we distinguish several cases for the choice of $\Theta_k^{\delta}(\delta)$.

\paragraph{Case 1}
First, we consider a piece-wise constant approximation, given by
\begin{align}
\Theta_k^{\delta}(\delta):=
\begin{cases}
1 &\text{if } \delta \in (\delta_{k-1},\delta_k] \text{ and } \alpha_{k-1}(\delta) > \beta_k(\delta)\\
1 &\text{if } \delta \in (\delta_{k},\delta_{k+1}] \text{ and } \alpha_k(\delta) \leq \beta_{k+1}(\delta)\\
0 &\text{else},
\end{cases}\label{kernel_approx_1}
\end{align}
where $\alpha_k(\delta):=\int_{\delta_{k}}^\delta\rho^{n-1}\kernelh(\rho)\d\rho$ and $\beta_k(\delta):=\int_{\delta}^{\delta_k}\rho^{n-1}\kernelh(\rho)\d\rho$.
That is, if $\kernel$ is given as a truncated fractional Laplace kernel~\eqref{kernel:FL_delta}, then $\alpha_k(\delta):=\frac{1}{2s}(\delta_{k}^{-2s}-\delta^{-2s})$ and $\beta_k(\delta):=\frac{1}{2s}(\delta^{-2s}-\delta_k^{-2s})$.
Thus, we approximate $\kernel(\delta)$ by $\kernel(\delta_k)$ if $\delta$ is sufficiently close to $\delta_k$, such that the integral of the kernel from $\delta$ to $\delta_k$ is smaller than that to the partition point on the other side.

\paragraph{{Case 2}}
To obtain an improved approximation quality, we also consider piece-wise linear approximation of the kernel. Here, we set
\begin{align}
\Theta_k^{\delta}(\delta):=
\begin{cases}
\frac{\delta-\delta_{k-1}}{\delta_k-\delta_{k-1}} &\text{if }\delta\in(\delta_{k-1},\delta_k]\\
\frac{\delta_{k+1}-\delta}{\delta_{k+1}-\delta_{k}} &\mbox{if}\;\;\;\delta\in(\delta_{k},\delta_{k+1}]\\
0 &\text{else}.
\end{cases}\label{kernel_approx_2}
\end{align}
Thus, $\Theta_k^\delta$ is given by the standard linear ``hat-functions'' on the grid $\delta_k$.

In both cases, for all $\delta\in[\deltamin,\deltamax]$, $u,v\in\V$, we define a parametrized bilinear form corresponding to the approximated kernel $\kerneld$ as follows
\begin{equation}
\ad(u,v;\delta)
:=\int_{\R^{2n}}\left(u(\x)-u(\xx)\right)\left(v(\x)-v(\xx)\right)\kerneld(\x,\xx;\delta)= \sum_{k=0}^K\Theta_k^{\delta}(\delta)\a(u,v;\delta_k).
\label{bil_form_mu}
\end{equation}
By the equivalence of nonlocal spaces~\eqref{equiv_nonloc_spaces} and the definition of $\kerneld$, we get that $\ad(u,v;\delta)$ is continuous and coercive on $\V\times\V$ with the continuity $\gamma_a$ and coercivity $\alpha_a$ constants defined as in~\eqref{constants_a}. 
Here, we use the ``partition of unity'' properties that $\Theta_k^\delta \geq 0$ and $\sum_k \Theta_k^\delta = 1$.

Next, we provide an estimate for the error caused by the approximation of the kernel $\kernel$ by $\kerneld$. In particular, we show that using \emph{Case~2} for sufficiently regular $u$, we obtain quadratic convergence in $\Delta\delta$, while \emph{Case~1} provides only a linear convergence order, but without any additional assumptions.

\begin{proposition}[\emph{Case 1}]\label{lemma:delta_conv}
Let $\kernel$ satisfy~\eqref{kernel_prop1}--\eqref{kernel_prop3} and let $\kerneld$ be defined as in~\eqref{kernel_approx_1}, then for any $\delta\in[\deltamin,\deltamax]$, and $u,v\in\V$ we obtain
\begin{equation}\label{error_bil_form}
\norms{\a(u,v;\delta)-\ad(u,v;\delta)}\leq C_a \, \Delta\delta \, \norm{u}_{L^2(\D)}\norm{v}_{L^2(\D)},
\end{equation}
where $C_a:= 4\omega_{n-1}C_\kernel$, with $\omega_{n-1}$ from~\eqref{volume_unit_sphere} and $C_\kernel$ from~\eqref{kernel_prop3}.
\end{proposition}
\begin{proof}
For simplicity of notation we introduce the following abbreviations: $u:=u(\x)$, $u^{\prime}:=u(\xx)$, and $\gamma:=\gamma(\x,\xx;\delta)$. Let $\delta\in[\delta_{k-1},\delta_k]$, $k=1,\dots,K$. Then, invoking the definition of $\kerneld$~\eqref{kernel_approx_gen} and~\eqref{kernel_prop1}--\eqref{kernel_prop3}, for any $u,v\in\V$, we obtain
\begin{multline*}
\norms{\a(u,v;\delta)-\ad(u,v;\delta)}
=\norms{\int_{S_{\delta_k}\setminus S_{\delta_{k-1}}}(u^{\prime}-u)(v^{\prime}-v)(\kernel-\kerneld)}\\
\leq 4\left(\int_{S_{\delta_k}\setminus S_{\delta_{k-1}}}u^2\norms{\kernel-\kerneld}\right)^{{1}/{2}}\left(\int_{S_{\delta_k}\setminus S_{\delta_{k-1}}}v^2\norms{\kernel-\kerneld}\right)^{{1}/{2}}\\
\leq 4\omega_{n-1}\left(\Theta_{k-1}^\delta(\delta)\int_{\delta_{k-1}}^\delta\rho^{n-1}\kernelh(\rho)\d\rho+\Theta_k^{\delta}(\delta)\int_{\delta}^{\delta_k}\rho^{n-1}\kernelh(\rho)\d\rho\right)\norm{u}_{L^2(\D)}\norm{v}_{L^2(\D)}\\
\leq 4\omega_{n-1} C_\kernel(\delta_k-\delta_{k-1})\norm{u}_{L^2(\D)}\norm{v}_{L^2(\D)}.
\end{multline*}
Taking the maximum over $k$ yields the desired result.
\end{proof}

For {\it Case~2} we obtain the following result.  
\begin{proposition}[\emph{Case 2}]\label{lemma:delta_conv2}
Let $\kernel$, $u$ be such that conditions of Theorem~\ref{thm:reg_delta} hold, and let $\kerneld$ be defined as in~\eqref{kernel_approx_2}, then for any $\delta\in[\deltamin,\deltamax]$, and $u,v\in\V$ we obtain that
\begin{equation}\label{error_bil_form2}
\norms{\a(u,v;\delta)-\ad(u,v;\delta)}\leq L_{a^{\prime}} (\Delta\delta)^2\,\norm{u}_{H^1_0(\D)}\norm{v}_{L^2(\D)},
\end{equation}
where $L_{a^{\prime}}$ is Lipschitz continuity constant defined in~\eqref{eq:lipschitz_der_a}.
\end{proposition}
\begin{proof}
Using fundamental theorem of calculus, we can express for a given $u,v\in\V$ and $\delta\in[\delta_{k-1},\delta_k]$:
\begin{equation}
\a(u,v;\delta)=\a(u,v;\delta_{k-1})+\int_{\delta_{k-1}}^{\delta}\a_\delta^{\prime}(u,v;\xi)\d\xi.\label{eq:proof1}
\end{equation}
From the definition of $\ad(u,v;\delta)$, with $\Theta_k^{\delta}(\delta)$ specified as in Case 2, it follows that
\begin{align}
\ad(u,v;\delta)&=\a(u,v;\delta_{k-1})+\Theta_{k}^{\delta}(\delta)\left(\a(u,v;\delta_k)-\a(u,v,\delta_{k-1})\right)\nonumber\\
&=\a(u,v;\delta_{k-1})+(\delta-\delta_{k-1})\frac{\a(u,v;\delta_k)-\a(u,v,\delta_{k-1})}{\delta_k-\delta_{k-1}}\nonumber\\
&=\a(u,v;\delta_{k-1})+\int_{\delta_{k-1}}^{\delta}a^{\prime}_{\delta}(u,v;\hat{\delta}_{u,v})\d\xi,\label{eq:proof2}
\end{align}
where $\hat{\delta}_{u,v}\in[\delta_{k-1},\delta_k]$, using the mean value theorem. Subtracting~\eqref{eq:proof1} from~\eqref{eq:proof2} and using Lipschitz continuity of $a^{\prime}_{\kernel}(u,v;\delta)$ in Theorem~\ref{thm:reg_delta}, we obtain
\begin{align*}
\norms{\a(u,v;\delta)-\ad(u,v;\delta)}&\leq\int_{\delta_{k-1}}^{\delta_k}\norms{\a_\delta^{\prime}(u,v;\xi)-\a_\delta^{\prime}(u,v;\hat{\delta}_{u,v})}\d\xi.
\end{align*}
Applying the estimate~\eqref{eq:lipschitz_der_a} for $a_\delta^{\prime}$ and using 
$\lvert{\hat{\delta}_{u,v}- \xi}\rvert \leq \Delta\delta$ concludes the proof.
\end{proof}

\subsection{Error due to the affine kernel approximation}
We consider the problem related to the affine kernel $\kerneld$: For $\delta\in[\deltamin,\deltamax]$, find $u(\delta)\in\V$, such that
\begin{equation}
\ad(u,v;\delta)=\dual{f,v},\quad\forall v\in\V,\label{var_trunc_delta}
\end{equation}
where $\ad$ is defined in~\eqref{bil_form_mu}. Now, using the results of Proposition~\ref{lemma:delta_conv2} and Proposition~\ref{lemma:delta_conv}, we provide the error in the solution caused by the affine approximation.
\begin{proposition}\label{prop:error_sol_affine_delta}
For $\delta\in[\deltamin,\deltamax]$, let $u(\delta),\widetilde{u}(\delta)\in V$ be the solutions of the problems~\eqref{nonl_linear_var},~\eqref{var_trunc_delta}, respectively. For {\it Case 1}, under conditions of Proposition~\ref{lemma:delta_conv}, and for \emph{Case 2}, under the conditions of Proposition~\ref{lemma:delta_conv2}, we obtain that
\begin{equation}
\norm{u(\delta)-\widetilde{u}(\delta)}_\V\leq\frac{C_P}{\alpha_a}\begin{cases}
{C_a\Delta\delta}\,\norm{u(\delta)}_{L^2(\D)} &\text{for {\it Case 1}},\\
{L_{a^{\prime}}(\Delta\delta)^2 }\,\norm{u(\delta)}_{H_0^1(\D)} &\text{for {\it Case 2}}.
\end{cases}\label{eq:error_sol_affine_delta}
\end{equation}
\end{proposition}
\begin{proof}
The proof can be obtained with standard methods, similarly as in Proposition~\ref{aposteriori_delta},
which will be provided later.
\end{proof}


\section{Regularity with respect to \texorpdfstring{$s$}{Lg}}\label{sec:regularity_s}
In this section we analyze the regularity of the solution of~\eqref{nonl_linear_var} with truncated fractional Laplace kernel~\eqref{kernel:FL_delta} with respect to the kernel parameter $s$. For simplicity of the presentation of our analysis we scale~\eqref{nonl_linear_var} by $c_{n,s}/2$. Then, for $s\in(0,1)$, $\delta\in(0,\infty]$, we seek $u\in\VD{s}$ such that for all $v\in\VD{s}$ 
\begin{align}
a(u,v;s)
:= \int_{S_\delta}\frac{\left(u(\x)-u(\xx)\right)\left(v(\x)-v(\xx)\right)}{\distn}\d(\xx,\x)
= \dual{f(s),v} \label{var_form_s}
\end{align}
where $f(s)= (2/c_{n,s}) F$ as defined in~\eqref{eq:RHS_s}. Throughout the next sections we also assume that $F\in H^{1/2-\varepsilon}(\D)$ for any $\varepsilon>0$.
For the sake of analysis, we consider first only the case of finite interaction radius $\delta < \infty$. The fractional Laplace case $\delta = \infty$ will be addressed later.

\begin{lemma}[Derivative of the bilinear form]\label{lemma:der_s_bilform}
Let $\delta\in(0,\infty)$ and $s_1,s_2\in (0,1)$ with $s\in(0,(s_1+s_2)/2)\subset(0,1)$.
Then, for $u\in\HD{s_2}$, $v\in\HD{s_1}$ the bilinear form $a(u,v;s)$ is infinitely many times differentiable, i.e., for $k=1,2,\dots$, there exist $a_s^{(k)}(u,v;s):=\frac{\d^k}{\d s^k}a(u,v;s)$, given by
\begin{align}
a_s^{(k)}(u,v;s):=(-2)^k\int_{S_{\delta}}\frac{\left(u(\x)-u(\xx)\right)\left(v(\x)-v(\xx)\right)\log^k(\dist)}{\distn}\d(\xx,\x).
\end{align}
Moreover, $a_s^{(k)}(u,v;s)$ is bounded and
\begin{equation}
\norms{a_s^{(k)}(u,v;s)}\leq C(k,\heps)\,\norm{u}_{\VD{s_2}}\norm{v}_{\VD{s_1}},\label{bound_der_a}
\end{equation}
where $C(k,\heps)=2^k\left(\left({k}/{(e\heps)}\right)^k+\delta^{\heps}(\log(\delta))^k_+\right)$ and $\heps=s_1+s_2-2s>0$.
\end{lemma}
\begin{proof}
The formal derivation of the derivative follows from a direct computation. In order to justify the differentiation under the integral, we can apply the dominated convergence theorem. For this purpose, we first prove the second statement. For $u\in\HD{s_2}$, $v\in\HD{s_1}$, using with the H{\"o}lder inequality we can estimate
\begin{multline}
\norms{a_s^{(k)}(u,v;s)}\leq 2^k\norms{\int_{S_{\delta}}\frac{\left(u(\x)-u(\xx)\right)\left(v(\x)-v(\xx)\right)\log^k(\dist)}{\dist^{n+2s+\heps}\dist^{-\heps}}\d(\xx,\x)}\\
\leq 2^k\sup_{\xi\in[0,\delta]}\frac{\norms{\log^k(\xi)}}{\xi^{-\heps}}\int_{S_{\delta}}
\frac{\norms{u(\x)-u(\xx)}\norms{v(\x)-v(\xx)}}{\dist^{n/2+s_2}\dist^{n/2+s_1}}\d(\xx,\x)\\
\leq 2^k \sup_{\xi\in[0,\delta]}\xi^{\heps}\lvert{\log^k(\xi)}\rvert \; \norm{u}_{\VD{s_2}}\norm{v}_{\VD{s_1}},
\end{multline}
where we have set $\xi = \dist \in [0,\delta]$.
Then, applying estimate~\eqref{estimate_log_k} from Appendix~\ref{app:aux} for the supremum, we obtain the desired result.
\end{proof}

Concerning the case $\delta = \infty$, we use the splitting
\begin{equation}
\label{split_bilinear}
a(u,v;\infty,s) = a(u,v;s) + C(\delta^\prime,n,s)(u,v)_{L^2(\D)},
\end{equation}
with $C(\delta^\prime,n,s) = (2\pi^{n/2})/(\Gamma(n/2)(\delta^\prime)^{2s}s)$, which is valid for all $\delta^\prime \geq \diam(\D)$; cf.~\cite{burkgunz2018} for a derivation of this simple identity. Then, a similar estimate as given above for $\delta < \infty$ can be given also in the fractional Laplace case.
\begin{corollary}
Using expression~\eqref{split_bilinear}, we obtain for any $\delta^\prime \geq \diam(\Omega)$ the derivative of the bilinear form corresponding to the fractional Laplacian $a(\cdot,\cdot;\infty,s)$ as
\begin{equation}
\frac{\d^k}{\d s^k}a(u,v;\infty,s)=\frac{\d^k}{\d s^k}a(u,v;
\delta^{\prime},s)+\frac{\d^kC(\delta^{\prime},n,s)}{\d s^k}(u,v)_{L^2(\D)}.
\end{equation}
Moreover, the estimate~\eqref{bound_der_a} remains valid with the constant defined as 
$C(k,\heps)=2^k\left(\left({k}/{(e\heps)}\right)^k+(\delta^\prime)^{\heps}(\log(\delta^\prime))^k_+ + C^2_P C^{(k)}_s(\delta^\prime,n,s)\right)$.
\end{corollary}
However, this estimate is only of minor relevance for the purposes of this paper, since an affine separable approximation of the bilinear form can be based directly on~\eqref{split_bilinear}, because the last term is already affine in the parameters, and only the truncated bilinear form for $\delta < \infty$ needs to be further approximated.

\begin{theorem}[$s$-regularity of the solution]
Let for any $s\in(0,1)$, fixed $\delta \in [\deltamin,+\infty]$, and $f$ as in~\eqref{eq:RHS_s} with $F \in H^{1/2-\varepsilon}(\D)$ for any $\varepsilon> 0$, $u(s)\in\VD{s}$ be the solution of~\eqref{var_form_s}. Then, $u$ is infinitely many times differentiable w.r.t.\ $s$ with values in $\VD{s}$. Thus, it holds $u\in C^{\infty}((0,1),L^2(\Omega))$.
\end{theorem}
\begin{proof}
First, select any $\tau^\star > 0$ with $s+\tau^\star < \min\{1,s+1/2-\varepsilon\}$.
By subtracting equations~\eqref{var_form_s} with $u(s)$ and with $u(s+\tau)$, both tested with $v\in\VD{s+\tau^{\star}}$, for some $0 < \tau \leq \tau^{\star}$, we obtain the following relation:
\begin{multline}
a\left(\frac{u(s+\tau)-u(s)}{\tau},v;s\right)
= \frac{1}{\tau}\dual{f(s+\tau)-f(s),v}\\
-\frac{1}{\tau}\big[a(u(s+\tau),v;s+\tau)-a(u(s+\tau),v;s)\big].\label{eq:sensitivity_limit}
\end{multline}
It is clear that $f(s)$ in~\eqref{var_form_s} is arbitrarily often differentiable in $s$, with all derivatives, denoted by $f_s^{(k)}(s)$, $k\in\mathbb{N}$, being in $H^{1/2-\varepsilon}(\D)$, $\varepsilon>0$.
Then, the right-hand side of~\eqref{eq:sensitivity_limit} is uniformly bounded in $\tau$. Indeed, taking $v = \d_\tau u = (u(s+\tau)-u(s))/\tau$, invoking the higher regularity of the solution from Theorem~\ref{thm:regularity_linear}, and using the fundamental theorem of calculus together with Lemma~\ref{lemma:der_s_bilform}, we obtain for all $\tau\leq\tau^\star/3$
\begin{equation*}
\norm{\d_\tau u}_{\VD{s}}\leq L_f + C(1,\tau^\star/3)\norm{u(s+\tau)}_{\VD{s+\tau^\star}}
\leq L_f + C\norm{f(s+\tau)}_{H^{1/2-\varepsilon}(\D)},
\end{equation*}
for $\varepsilon>0$ sufficiently small. Here, $L_{f}$ is the Lipschitz constant of $f(s)$.
Now letting $\tau\to 0$ in~\eqref{eq:sensitivity_limit} we obtain \(\d_t u \rightharpoonup u^{\prime}_s\) together with the following sensitivity equation:
\begin{equation}
a(u^{\prime}_s(s),v;s)=f^{\prime}_s(v;s)-a^{\prime}_s(u(s),v;s),\quad\forall v\in\VD{s+\tau^{\star}}.
\label{eq:sensitivity_s}
\end{equation}
By the density of $\VD{s+\tau^{\star}}$ in $\VD{s}$, we obtain that the above equation also holds for all $v\in\VD{s}$.
From Theorem~\ref{thm:regularity_linear}, we have that $u(s)\in\HD{s+1/2-\varepsilon}\subset\HD{s+\tau^\star}$.
Now, we potentially decrease $\tau^\star$, such that additionally $s - \tau^\star/2 > 0$.
Applying~\eqref{bound_der_a} with $s_1 = s - \tau^\star/2$, $s_2 = s + \tau^\star$, we obtain that
\begin{equation*}
\norms{a^{\prime}_s(u(s),v;s)}\leq C(1,\tau^\star/2)\norm{u(s)}_{\HD{s+\tau^\star}}\norm{v}_{\HD{s-\tau^\star/2}},
\end{equation*}
which also implies that $a^{\prime}_s(u(s),\cdot;s)\in H^{-s+\tau^\star/2}(\D)$.
Then, by denoting by $R(v)$ the right-hand side of~\eqref{eq:sensitivity_s}, we obtain that $R\in H^{-s+\tau^\star/2}(\D)$.
Applying Theorem~\ref{thm:regularity_linear}, it follows that $u^{\prime}_s(s)\in\HD{s+\tau^\star/2}$.
Similarly, the second sensitivity $u^{(2)}_s(s)$ fulfills
\begin{equation}
a(u_s^{(2)}(s),v;s)=f^{(2)}_{s}(v;s)-2a^{\prime}_s(u^{\prime}_s(s),v;s)-a^{(2)}_s(u(s),v;s),\quad\forall v\in\VD{s}.\label{eq:sensitivity_s2}
\end{equation}
Then, applying~\eqref{bound_der_a} for $a^{\prime}_s(u^{\prime}_s(s),v;s)$ with $s_1=s-\tau^\star/4$, $s_2=s+\tau^\star/2$ and for $a^{(2)}_s(u(s),v;s)$ with $s_1=s-\tau^\star/4$, $s_2=s+\tau^\star$, we obtain that the right-hand side of~\eqref{eq:sensitivity_s2} is in $H^{-s+\tau^\star/4}(\D)$. Again, applying regularity result~\eqref{nonl_regularity}, we conclude that $u^{(2)}_s(s)\in\HD{s+\tau^\star/4}$. Proceeding iteratively and using an induction argument, it is easy to show that $k$-th derivative $u^{(k)}_s(s)\in\VD{s}$, is the unique solution of
\begin{equation*}
a(u^{(k)}_s(s),v;s)=f^{(k)}_s(v;s)-\sum_{j=1}^k\binom{k}{j}a^{(j)}_s(u^{(k-j)}_s(s),v;s),\quad\forall v\in\VD{s},
\end{equation*}
and, in addition, $u^{(k)}_s(s)\in\HD{s+\tau^\star/2^k}$. By embedding the solution spaces into the common space $L^2(\Omega)$, we obtain the last property.
\end{proof}


\section{Affine approximation with respect to \texorpdfstring{$s$}{Lg}}\label{sec:affine_s}

Approximation with respect to $s$ poses similar difficulties as for the case of $\delta$. However, here, invoking the higher regularity properties, we can approximate the bilinear form with high order polynomials. Concretely, we adopt the Chebyshev interpolation for $a(\cdot,\cdot;s)$ with respect to $s$ on the interval $[\smin,\smax]\subset(0,1)$.

For $m\in\mathbb{N}$, we consider the following partitioning of the interval $[\smin,\smax]$:
$0<\smin:=s_0<s_1<\cdots<s_{M}:=\smax<\infty$, where $s_m$ are the Chebyshev maximal points on the interval $[\smin,\smax]$, given by $s_m = (1/2)(\smin + \smax) - (1/2)(\smax-\smin) \cos((m/M) \, \pi)$, $m=1,\dots,M$. Then, for a given $\delta\in(0,\infty)$, $s\in(0,1)$, and sufficiently regular $u,v$, we approximate 
\begin{align}
a(u,v;s)\approx\adM(u,v;s):=\sum_{m=0}^M\Theta_m^s(s)a(u,v;s_m),\quad \Theta_m^s(s):=\prod_{\substack{j=0\\j\neq m}}^{M}\frac{s-s_j}{s_m-s_j}.\label{bilform_sapprox}
\end{align}
Note that, in order for $\adM$ to be well-defined, the regularity that both $u$ and $v$ are in $\HD{s}$ is not sufficient, but, for instance, taking $u,v\in\HD{\smax}$ is sufficient. We will discuss this more thoroughly in the following.
\begin{remark}
We note that for the fractional Laplace case $\delta = \infty$, we can obtain an affine separable bilinear form by combining the splitting~\eqref{split_bilinear} for a choice of $\delta \geq \diam(\Omega)$ with the approximation~\eqref{bilform_sapprox}. 
\end{remark}
Thus, we will restrict attention to the case $\delta < \infty$ in the following.
\begin{lemma}[Interpolation error]\label{lemma:s_int_error}
Let $u\in\HD{s_2}$, $v\in\HD{s_1}$ and $s\in[\smin,\smax]\subset(0,1)$, such that $(s_1+s_2)/2-1/2<\smax<(s_1+s_2)/2$, $s_1,s_2\in(0,1)$. Then, for $\delta\in(0,\infty)$, we obtain
\begin{equation}
\norms{a(u,v;s)-\adM(u,v;s)}\leq\sigma^{M+1}C(\delta)\norm{u}_{\VD{s_2}}\norm{v}_{\VD{s_1}},
\label{eq:s_error}
\end{equation}
where $\sigma=(\smax-\smin)/(2\heps(\smax))$, $C(\delta)=4(e^{-1}+\delta^{\heps(\smin)+1})$ if $\delta>1$ and $C(\delta)=4e^{-1}$ if $\delta\leq 1$, where $\heps(s)=s_1+s_2-2s$.
\end{lemma}
\begin{proof}
We present a proof for $\delta>1$, the case $\delta\leq 1$ follows along the same lines. Using the Chebyshev polynomial interpolation error estimate, and the bound~\eqref{bound_der_a}, we can estimate that
\begin{multline*}
\norms{a(u,v;s)-\adM(u,v;s)}\leq\frac{(\smax-\smin)^{M+1}}{2^{2M}(M+1) !}\max_{\xi\in[\smin,\smax]}\norms{a_s^{(M+1)}(u,v;\xi)}\\
\leq\frac{(\smax-\smin)^{M+1}}{2^{2M}(M+1) !}\norm{u}_{\VD{s_2}}\norm{v}_{\VD{s_1}}\max_{\xi\in[\smin,\smax]}\norms{C(M+1,\heps(\xi))}.
\end{multline*}
Using the fact that $e\left(\frac{M+1}{e}\right)^{M+1}\leq (M+1) !$ and $\log^{M+1}(\delta)\leq\delta(M+1)!$ (cf.\ Proposition~\ref{appx: prop1}), we can estimate the last term in the previous inequality as follows
\begin{multline*}
\max_{\xi\in[\smin,\smax]}\norms{C(M+1,\heps(\xi))}
\leq 2^{M+1}(M+1) !\max_{\xi\in[\smin,\smax]}\left(\frac{1}{e\heps(\xi)^{M+1}}+\delta^{\heps(\xi)+1}\right)\\
\leq 2^{M+1}(M+1) !\left(\frac{1}{e\heps(\smax)^{M+1}}+\delta^{\heps(\smin)+1}\right)
\leq 2^{M+1}(M+1) !\frac{e^{-1}+\delta^{\heps(\smin)+1}}{\heps(\smax)^{M+1}}.
\end{multline*}
A combination of the estimates concludes the proof.
\end{proof}

\begin{remark}
\label{rem:restriction_s}
For $\sigma < 1$, i.e., $\heps(\smax) > (\smax - \smin)/2$,
the error term in~\eqref{eq:s_error} converges to zero exponentially.
This can be guaranteed in two situations, where we assume that $v\in\HD{\smin}$, $u\in\HD{\smin + 1/2 - \varepsilon}$ for arbitrarily small $\varepsilon > 0$.
\begin{itemize}
\item
  For $\smin \leq 1/2$ and $\smax - \smin < 1/5$: In this case, we can choose
  \[
    s_1 = \smin, \quad s_2 = \smin + 1/2 - \varepsilon,
  \]
  which leads to $\sigma < 1$
   in the case of $\varepsilon < 1/2 - (5/2)(\smax - \smin)$. 
\item
  For $\smin > 1/2$ and $\smax - \smin < (2/3)(1-\smax)$: Here, a choice of $s_2$ as above is prevented by the restriction $s_2 < 1$. Thus, we chose
  \[
    s_1 = \smin, \quad s_2 = 1 - \varepsilon,
  \]
  which is possible using $\HD{\smin+1/2-\varepsilon} \hookrightarrow \HD{1-\varepsilon}$.
  Then, we obtain $\sigma < 1$
  in the case of $\varepsilon < 1 - \smax - (3/2)(\smax - \smin)$. 
\end{itemize}
\end{remark}
In the following, we will assume that one of the conditions of Remark~\ref{rem:restriction_s} holds, which limits the size of the interval $[\smin,\smax]$. However, in the case that the interval of interest is larger, we can simply subdivide it into several subintervals $[\smin^k,\smax^k]$ with $\smin^0 = \smin$, $\smax^{k-1} = \smin^k$ for $k = 1,2,\ldots,K$, and $\smax^K = \smax$, such that each subinterval fulfills the conditions of Remark~\ref{rem:restriction_s}. Furthermore,  for any $\smin, \smax \in (0,1)$ the interval $[\smin, 2/3]$ can be covered by at most four subintervals, and the remaining interval $[2/3, \smax]$ can be covered by a finite number of subintervals. However, for notational simplicity we will simply accept the restrictions from Remark~\ref{rem:restriction_s} and consider a single interval.

\subsection{Error due to the affine kernel approximation}
In the following, we analyze the error cause by the approximation~\eqref{bilform_sapprox}.
Due to the fact that the coefficents $\Theta^s_m$ can be negative, we can not guarantee that \(\adM\) is coercive on $\HD{s}\times\HD{s}$, $s\in(0,1)$. To overcome this difficulty, we introduce a regularized problem. 

In what follows, we assume that the conditions from Remark~\ref{rem:restriction_s} are fulfilled.
We also define the mean value for the choices of $s_1$ and $s_2$ given there as
\[
\s := (s_1+s_2)/2 = \begin{cases}
  \smin + 1/4 - \varepsilon/2 &\text{if } \smin \leq 1/2, \\
  (\smin + 1)/2 - \varepsilon/2 &\text{else}.
\end{cases}
\]
For a regularization parameter $\rho>0$, we define $\adr\colon\VD{\s}\times\VD{\s}\times[\smin,\smax]\to\R$ as
\begin{equation}
\label{bilinear_reg}
\adr(u,v;s):=\adM(u,v;s)+\rho(u,v)_{\VD{\s}}. 
\end{equation}
We consider the following regularized problem: Find $u\in\HD{\s}$, such that
\begin{equation}
\adr(u^\rho,v;s)=\dual{f(s),v} \qquad\forall v\in\HD{\s}.\label{detailed_reg}
\end{equation}
\begin{proposition}
Under the assumption $\rho > C(\delta)\sigma^{M+1}$,
the bilinear form~\eqref{bilinear_reg} is coercive and continuous on $\VD{\s}\times{\VD{\s}}$,
with coercivity constant $\alpha_a^{\rho}:=\rho-C(\delta)\sigma^{M+1}$ and continuity constant $\gamma_a^{\rho}:={\sconst{s,\s}^2}+\rho+C(\delta)\sigma^{M+1}$.
Additionally, it is coercive on $\VD{s}\times{\VD{s}}$ for any $s \in [\smin,\smax]$ with constant one.
Thus, there exists a unique solution $u^\rho \in \VD{\s} \hookrightarrow \VD{s}$ of~\eqref{detailed_reg} with
$
\norm{u^\rho}_{\VD{s}} \leq \norm{f(s)}_{\VD{s}'}.
$
\end{proposition}
\begin{proof}
Using the error estimate~\eqref{eq:s_error} with the choices $s_1 = s_2 = \s$, we can write
\begin{align*}
\adr(u,u;s)
&\geq \a(u,u;s) - \norms{\a(u,u;s) - \adr(u,u;s)} + \rho\norm{u}^2_{\VD{\s}}\\
&\geq\norm{u}^2_{\VD{s}}+(\rho-C(\delta)\sigma^{M+1})\norm{u}^2_{\VD{\s}}.
\end{align*}
By neglecting either the first or the second term, we obtain the coercivity of $\adr$ as stated.
Second, using the Sobolev embedding~\eqref{equiv_2}, we obtain
\begin{align*}
\norms{\adr(u,v;s)}&\leq
\norm{u}_{\VD{s}}\norm{v}_{\VD{s}}+C(\delta)\sigma^{M+1}\norm{u}_{\VD{\s}}\norm{v}_{\VD{\s}}
+\rho\norm{u}_{\VD{\s}}\norm{v}_{\VD{\s}}\\
&\leq\left({\sconst{s,\s}^2}+\rho+C(\delta)\sigma^{M+1}\right)\norm{u}_{\VD{\s}}\norm{v}_{\VD{\s}}.
\end{align*}
The existence and uniqueness of the solution follows now with the Lax-Milgram lemma. The given estimate is due to the coercivity of $\adr$ on $\VD{s}$.
\end{proof}

Additionally, we can quantify the error between the solutions of~\eqref{detailed_reg} and~\eqref{var_form_s}. 
\begin{proposition}
\label{prop:error_s_approx}
Let $u\in\VD{s}$, $u^{\rho}\in\VD{\s}$ are the solutions of~\eqref{var_form_s} and~\eqref{detailed_reg}, respectively. Then, we have the following error bound
\begin{equation}
\norm{u^{\rho}-u}_{\VD{s}}\leq {\sconstv{\smin,s}} \left(C(\delta)\sigma^{M+1}+\rho\right)\norm{u}_{\VD{s_2}},
\label{error_reg_detailed}
\end{equation}
where $s_2 = \min\{\;\smin+1/2,\,1\;\} -\varepsilon$.
\end{proposition}
\begin{proof}
From Theorem~\ref{thm:regularity_linear}, we have that $u\in\HD{\smin+1/2-\varepsilon}\subset\HD{\s}$.  Then for $v\in\HD{\s}\subset\HD{\smin}$ and using the Cauchy-Schwarz inequality, we can estimate 
\[
(u,v)_{\VD{\s}} = \int_{S_\delta}  \frac{(u(\x) - u(\xx))(v(\x) - v(\xx))}{\dist^{n/2 + s_1}\dist^{n/2 + s_2}} \d(\x,\xx)
\leq\norm{u}_{\VD{s_2}}\norm{v}_{\VD{s_1}},
\]
with $s_1 = \smin$ and $s_2 = \min\{\;\smin+1/2-\varepsilon,\,1-\varepsilon\;\}$ (cf.\ Remark~\ref{rem:restriction_s}).
Moreover, for $v\in\HD{\s}$ we obtain that
\begin{equation*}
\adr(u^{\rho}-u,v)=a(u,v)-\adM(u,v)-\rho(u,v)_{\VD{\s}}.
\end{equation*}
Then, taking $v=u^{\rho}-u\in\HD{\s}$, using coercivity of $\adr$ on $\VD{s}\times\VD{s}$, the estimate above, and Lemma~\ref{lemma:s_int_error} with the same $s_1$, $s_2$ as above, we obtain
\begin{multline*}
\norm{u^{\rho}-u}_{\VD{s}}^2\leq C(\delta)\sigma^{m+1}\norm{u}_{\VD{s_2}}\norm{u^{\rho}-u}_{\VD{\smin}}+\rho\norm{u}_{\VD{s_2}}\norm{u^{\rho}-u}_{\VD{\smin}}\\
\leq{\sconstv{\smin,s}}\left(C(\delta)\sigma^{m+1}+\rho\right)\norm{u}_{\VD{s_2}}\norm{u^{\delta}-u}_{\VD{s}},
\end{multline*}
which yields the desired result.
\end{proof}
\begin{corollary}
For a choice of $\rho = 2 C(\delta)$, the problem~\eqref{detailed_reg} is well posed, and admits the estimate
\[
\norm{u^{\rho}-u}_{\VD{s}}\leq C \sigma^{M+1} \norm{u}_{\VD{s_2}},
\]
with $C$ depending only on $s$, $\smin$, $\smax$, and $\delta$, with $\sigma$ as given in Lemma~\ref{lemma:s_int_error}.
\end{corollary}


\section{Reduced basis approximation}\label{sec:rbm}

Solving the problems~\eqref{nonl_linear_var} for various values of the parameters, e.g., by finite elements or other discretization methods, would require a substantial computational effort due to the fact that the underlying discrete system will consist of banded matrices with  bandwidth related to the nonlocal interaction radius $\delta$. For $\delta=\infty$, as in the example of the fractional Laplacian, we need to deal with full matrices. As a remedy, one could use the affine problems~\eqref{var_trunc_delta} and~\eqref{detailed_reg} instead, to reduce the computational cost of assembling the matrices for different parameters. However, in this case one still has to confront the problem of solving dense high-dimensional systems.

In this section, we describe how to build the reduced basis approximation upon problems~\eqref{var_trunc_delta} and~\eqref{detailed_reg} -- which are often referred as the ``detailed problems'' -- and to obtain a reduced problem of much smaller dimension. 
We note that~\eqref{var_trunc_delta} and~\eqref{detailed_reg} can also be replaced by high-fidelity discrete problems, where $\V$, $\VD{s}$ are substituted with high-dimensional discrete approximation spaces, e.g., finite element spaces. 
Subsequently, the reduced model is constructed based on an approximation of the detailed spaces
by low-dimensional reduced spaces $\V_N\subset\V$, $\VN{s}\subset\VD{s}$. 
Various approaches exist for the construction of the reduced basis approximation
spaces.
Typically, this is done by applying an iterative greedy strategy to a set of snapshots, i.e., solutions computed for different parameter values; see, e.g.,~\cite{buffa,binev}. 
The computational speed-up is then achieved by an offline-online computational procedure, invoking the affine-parameter dependency of the problem: In the offline routine, which is performed only once, we assemble all parameter independent forms needed in the construction of the affine approximation, i.e., we assemble $a(\cdot,\cdot;\delta_k)$ in~\eqref{bil_form_mu} and $a(\cdot,\cdot;s_m)$ in~\eqref{bilform_sapprox}, for $k=0,\dots,K$, $m=0,\dots,M$. During the online stage, we evaluate the parameter-dependent components, i.e., $\Theta_k^{\delta}(\delta)$ and $\Theta_m^s(s)$, and solve the corresponding reduced system. This stage is executed multiple times for each new parameter value. 

\subsection{RBM with respect to \texorpdfstring{$\delta$}{Lg}}
The RBM approximation for $\delta$ is rather straightforward and we outline it only briefly. 
Using a greedy algorithm, we construct the reduced bases spaces $\V_N\subset\V$ from the set of snapshots $\{u(\delta_i), i=1,\dots,N\}$, where $u(\delta_i)$ are the solutions of~\eqref{nonl_linear_var} for different $\delta_i\in\P^{\delta}$, $\P^{\delta}:=[\deltamin,\deltamax]$. To find the reduced solution we solve the following problem: For $\delta\in\P^{\delta}$, find $u_N(\delta)$, such that
\begin{equation}
\ad(u_N,v;\delta)=\dual{f,v}\qquad \forall v\in\V_N,\label{eq:rb_problem_delta}
\end{equation}
where, we recall, $\ad(\cdot,\cdot;\delta)$ is defined in~\eqref{bil_form_mu} either using {\it Case 1} or {\it Case 2} approximations. To validate the error caused by the RB approximation, we derive the corresponding a~posteriori error estimates. We define the residual $r(\cdot;\delta)\in\Vd$ by
\[
r(v;\delta):=\ad(\tr{u}_N,v;\delta)-\dual{f,v},\quad v \in \V.
\]
We note that the residual vanishes on $\V_N$, i.e., $r(v;\delta)=0$ for any $v\in V_N$. Then, we obtain the following error bounds.

\begin{proposition}[A posteriori error estimator for $\delta$]
\label{aposteriori_delta}
For any $\delta\in\P^{\delta}$ the reduced basis error for {\it Case~1}, under conditions of Proposition~\ref{lemma:delta_conv}, and for {\it Case~2}, under the conditions of Proposition~\ref{lemma:delta_conv2}, we obtain
\begin{equation}
\norm{\tr{u}_N-u}_\V\leq
\frac{\norm{{r}}_{\Vd}}{\alpha_a}+\frac{C_P}{\alpha_a}\begin{cases}
{C_a}\Delta\delta\,\norm{\tr{u}_N}_{L^2(\D)} &\text{ for {\it Case 1}},\\
{L_a^{\prime}}(\Delta\delta)^2\,\norm{\tr{u}_N}_{H_0^1(\D)} &\text{ for {\it Case 2}},
\end{cases}\label{eq:apos_delta}
\end{equation}
where $\alpha_a$, $C_a$, $C_P$, and $L_{a^\prime}$ are from~\eqref{constants_a},~\eqref{error_bil_form},~\eqref{poincare}, and~\eqref{eq:lipschitz_der_a}, respectively.
\end{proposition}
\begin{proof}
For any $v\in\V$ we have
\begin{align*}
\a(\tr{u}_N-u,v)
= r(v) + \a(\tr{u}_N,v) - \ad(\tr{u}_N,v).
\end{align*}
Taking $v=\tr{u}_N-u$, using coercivity of $\a$, we obtain
\begin{align*}
\alpha_a\norm{\tr{u}_N-u}^2_{\V}
\leq \norm{{r}}_{\Vd}\norm{\tr{u}_N-u}_{\V} + \norms{\a(\tr{u}_N,\tr{u}_N-u)-\ad(\tr{u}_N,\tr{u}_N-u)}
\end{align*}
and applying the error bound~\eqref{error_bil_form}, resp.~\eqref{error_bil_form2}, we conclude the proof.
\end{proof}

\subsection{RBM with respect to \texorpdfstring{$s$}{Lg}}

Upon the detailed problem~\eqref{detailed_reg}, we build the corresponding reduced basis approximations. For suitably chosen approximation spaces $\VN{s}$, we have the following problem:
For a given $s\in\P^s:=[\smin, \smax ]$, find $u^{\rho}_N(\parmu)\in\VN{s}$, such that
\begin{equation}
\adr(u^{\rho}_N,v;s)=\dual{f(s),v},\quad\forall v\in\VN{s}.\label{rb_reg}
\end{equation}
In particular, we chose the reduced space as follows:
\begin{equation}
\VN{s}:=\left\{\spann\{u(s_i),\; s_i\in\Ps,\; i=1,\dots,N\},\;\;\norm{\cdot}_{\VD{s}}\right\}.
\end{equation}
We note that, invoking Theorem~\ref{thm:regularity_linear}, we have that
\[
\spann\{u(s_i),\; s_i\in\Ps,\; i=1,\dots,N\}\subset\HD{\smin+1/2-\varepsilon}\subset\HD{\s}\subset\HD{s},
\]
and, hence, $\VN{s}\subset\VD{s}$. This also guarantees the well-posedness of the reduced problem~\eqref{rb_reg}.
For any $s\in\P^s$, we define the residual 
\begin{equation}
r^{\rho}(v;s):=\adr(u_N^{\rho},v;s)-\dual{f(s),v},\quad\forall v\in\VD{\s}.
\end{equation}
We note that $r^{\rho}(v;s)$ vanishes for any $v\in\VN{\s}$. Then, we derive the a~posteriori error bound associated with the reduced basis approximation.
\begin{proposition}[A posteriori error estimator for $s$]
Let for $s\in\P^s$, $u\in\VD{s}$, $u_N^{\rho}(s)\in\VN{s}$ be a solution of~\eqref{var_form_s} and~\eqref{rb_reg}, respectively. Then, with 
$s_2 = \min\{\smin+1/2,\, 1\;\} - \varepsilon$, the reduced basis error can be estimated as
\begin{equation}
\norm{u^{\rho}_N-u}_{\VD{s}}\leq{\sconstv{\smin,s}}\left(\norm{r^{\rho}}_{\VD{-\smin}}+\left(\rho+C(\delta)\sigma^{M+1}\right)\norm{u_N^{\rho}}_{\VD{s_2}}\right).\label{eq:apos_s}
\end{equation}
\end{proposition}
\begin{proof}
For $v\in\VD{\s}$, we obtain
\begin{align*}
a(u_N^{\rho}-u,v)&=\adr(u_N^{\rho},v)-\dual{f,v}+a(u_N^{\rho},v)-\adr(u_N^{\rho},v)\\
&=r^{\rho}(v;s)+a(u_N^{\rho},v)-\ad(u_N^{\rho},v)-\rho(u_N^{\rho},v)_{\VD{\s}}.
\end{align*}
Taking $v=u_N^{\rho}-u\in\VD{\s}\subset\VD{\smin}$, we obtain by similar arguments as in Proposition~\ref{prop:error_s_approx}
\begin{multline*}
\norm{u_N^{\rho}-u}^2_{\VD{s}}\leq\norm{r^{\rho}}_{\VD{-\smin}}\norm{u_N^{\rho}-u}_{\VD{\smin}}\\
+C(\delta)\sigma^{M+1}\norm{u_N^{\rho}}_{\VD{s_2}}\norm{u_N^{\rho}-u}_{\VD{\smin}}
+\rho\norm{u_N^{\rho}}_{\VD{s_2}}\norm{u_N^{\rho}-u}_{\VD{\smin}}\\
\leq
{\sconstv{\smin,s}}\left(\norm{r^{\rho}}_{\VD{-\smin}}+\left(\rho+C(\delta)\sigma^{M+1}\right)\norm{u_N^{\rho}}_{\VD{s_2}}\right)\norm{u_N^{\rho}-u}_{\VD{s}}.
\end{multline*}
Dividing by the error yields the result.
\end{proof}

We point out that the derived a-posteriori error bounds~\eqref{eq:apos_delta} and~\eqref{eq:apos_s} can be computed efficiently in an offline-online manner. 
This means that parameter-independent components for the computation of the error bound can be precomputed in the offline stage, and the online evaluation of the error bound depends only on the reduced dimension $N$ and $K$, respectively $M$.


\section{Numerical results}\label{sec:numerics}

We consider $\D = (0, 1)$, which is discretized with the uniform mesh of mesh size $h$, which we set to $h = 2^{-9}$. The interaction domain is given by $\DI=(-\delta,0)\cup(1,1+\delta)$, with $\delta\in[\deltamin,\deltamax]=[0.0625,1]$.
We consider the case of the fractional Laplace type kernels, that is kernels parametrized by $\delta$ and $s$ and defined in~\eqref{kernel:FL_delta}. The parameter domain for $\delta$ and $s$ is set to 
\[\P:=\P^\delta\times\P^{s}=[\deltamin,\deltamax]\times[\smin,\smax]=[0.0625,1]\times[1/3,1/2].\]
For the construction the RB method for the parameter $\delta$ we consider a training set $\P^{\delta}_{\rm train}$,  consisting of 121 uniformly distributed points with a step size $\Delta\delta=2^{-7}$. Analogously, for the parameter $s$, we use a training set $\P^{s}_{\rm train}$, comprising of 50 points. To validate the RB approximation we use a testing sets $\P^{\delta}_{\rm test}$ and $\P^{s}_{\rm test}$ consisting of 100 and 30 randomly distributed points on $\P^{\delta}$ and $\P^{s}$, respectively. 

In Figure~\ref{fig:0} we plot the variability of the discrete solution of~\eqref{nonl_linear_var} w.r.t.\ either $\delta$ or $s$ and $f(s) = (2/c_{n,s})F$ with $F$ chosen as a characteristic function on $[1/2,1]$, i.e., $F=\chi_{[1/2,1]}$. We observe that even having only one parameter varying in the model, there is a non-trivial parameter dependency of the solution. Moreover, we observe how the regularity of the solution deteriorates for decreasing $s$, as indicated by the theory.
\begin{figure}[htbp]
\centering
\includegraphics[width=0.42\textwidth]{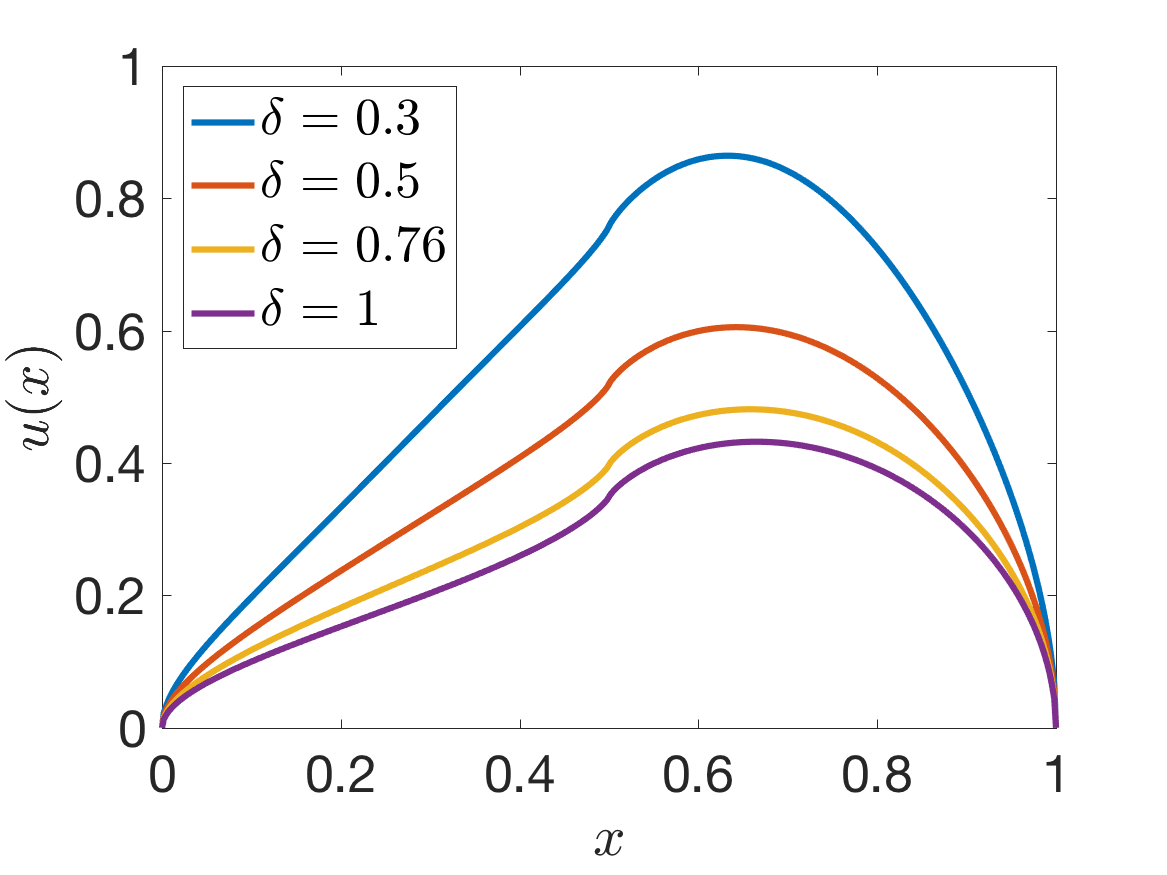}
\includegraphics[width=0.42\textwidth]{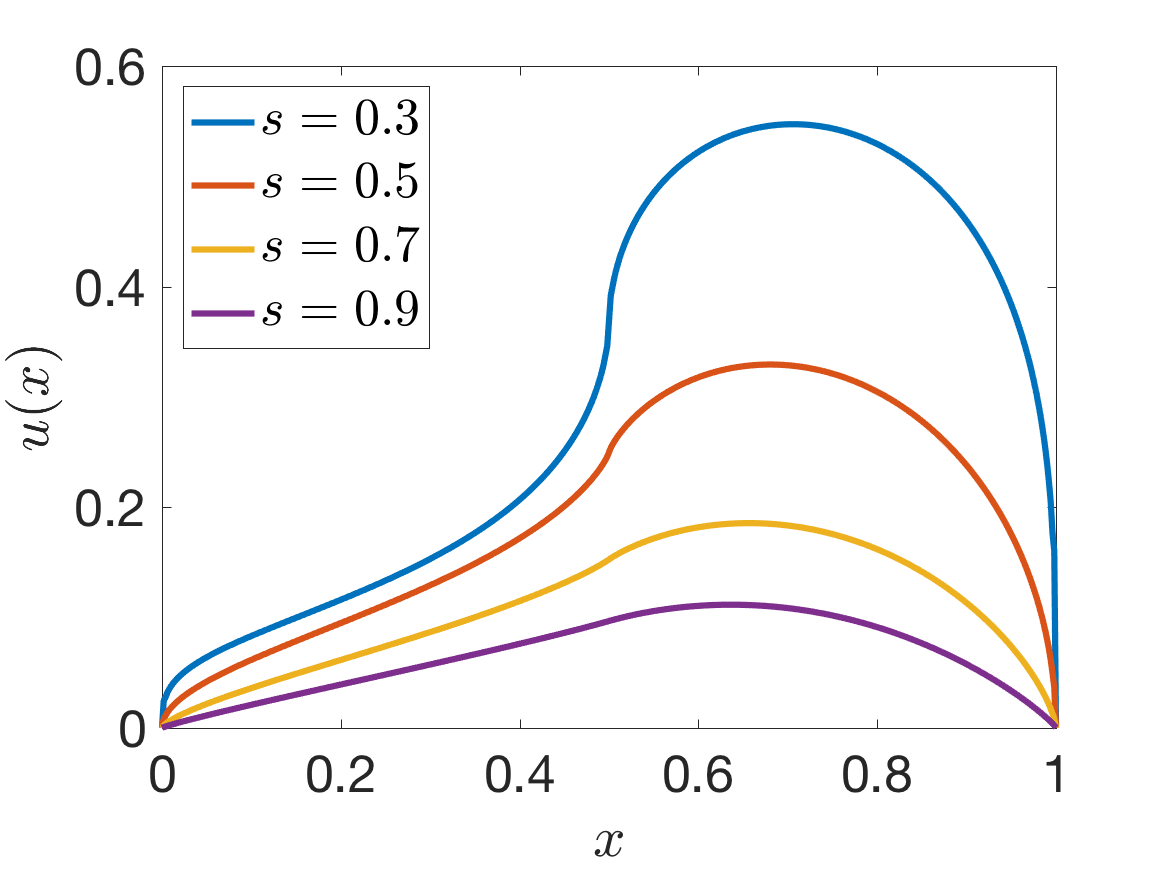}
\caption{Snapshots of the solution for different values of $\delta$ with $s=0.5$ (left) and different $s$ with $\delta=+\infty$ (right).}
\label{fig:0}
\end{figure}

\subsection{Affine approximation}

First, we numerically investigate the errors solely due to the the affine approximation of the bilinear form, without any reduced basis approximation. 
In the following, unless otherwise stated, the pivot energy space $\V$ is chosen for $\delta^{\star}=1/2$ and the data term is chosen as $f(s) = (2/c_{n,s})F$, where $F \equiv -1$. We note that we do not consider the error between continuous and the finite element solution, and only compare the errors in terms of the discrete solution. Hence, for notational convenience we denote the discrete solution using the same notation as for the continuous solution.

In Figure~\ref{fig:1} we depict the pointwise errors 
$\norm{u(\delta)-\widetilde{u}(\delta)}_{\V}$ for $\delta\in\P_{\rm train}^{\delta}$
between discrete solutions of the original problem~\eqref{nonl_linear_var} and the problem~\eqref{var_trunc_delta} with an affine kernel $\kerneld$ for values of $K\in\{\;5,9,16,31,61\;\}$ in~\eqref{kernel_approx_gen} using a uniform partitioning of $[\deltamin,\deltamax]$. 
We clearly observe the reduction of the error while increasing $K$ and smaller errors for \emph{Case~2} in comparison to \emph{Case~1}. 
In addition, we notice that the maximal error is not uniformly distributed over the intervals and significantly increases for smaller values of $\delta$. This suggests that using a uniform partitioning of $[\deltamin,\deltamax]$ for the construction of~\eqref{kernel_approx_gen} may not be optimal. {In fact, from the estimate~\eqref{eq:error_sol_affine_delta}, taking into account the particular form of the fractional Laplace kernel~\eqref{kernel:FL_delta}, we can deduce that distributing 
\begin{equation}
\label{delta_graded}
\delta_k \sim \deltamin\left(\frac{\deltamax}{\deltamin}\right)^{{k}/{K}},\quad k=0,1,\ldots,K,
\end{equation}
better equalizes the error in terms of the constant appearing in the provided estimate. }
Using such partitioning for the construction of $\kerneld$ for \emph{Case~2}, we measure again the corresponding errors in the discrete solutions, which is depicted in Figure~\ref{fig:1}. Here, we observe an additional reduction in the maximal error over \emph{Case~2} with the uniform refinement, and, in addition, a better equilibration of the maximal error over the subintervals $[\delta_{k-1},\delta_k]$. 
\begin{figure}[htbp]
 \centering
\includegraphics[width=0.32\textwidth]{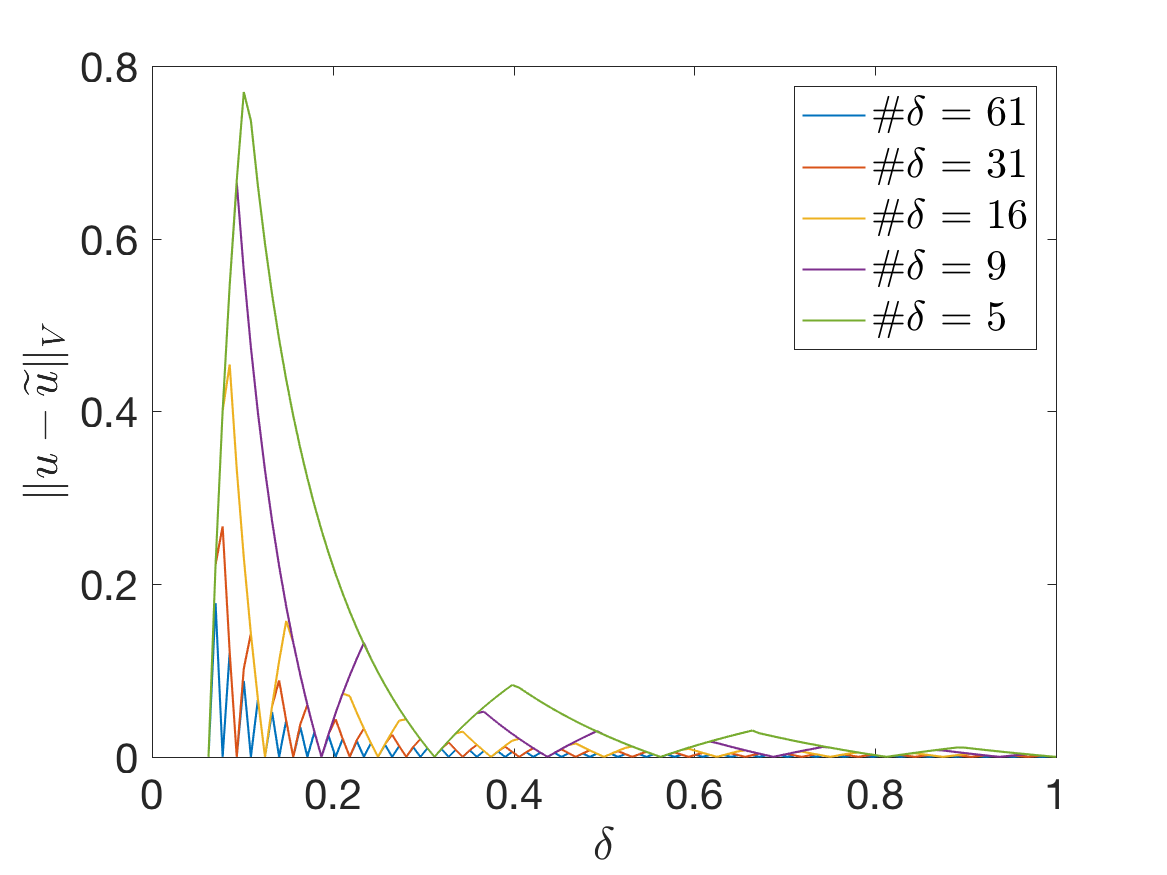}
\includegraphics[width=0.32\textwidth]{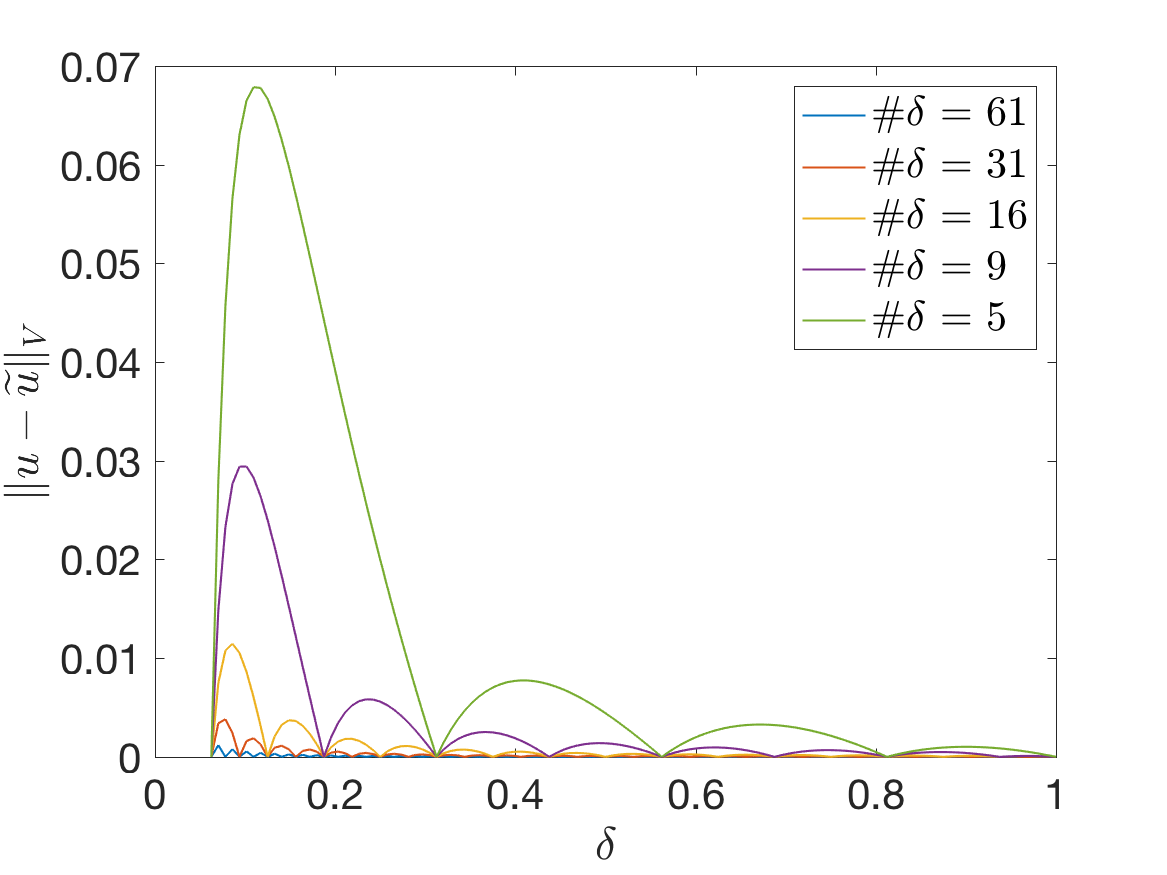}
\includegraphics[width=0.32\textwidth]{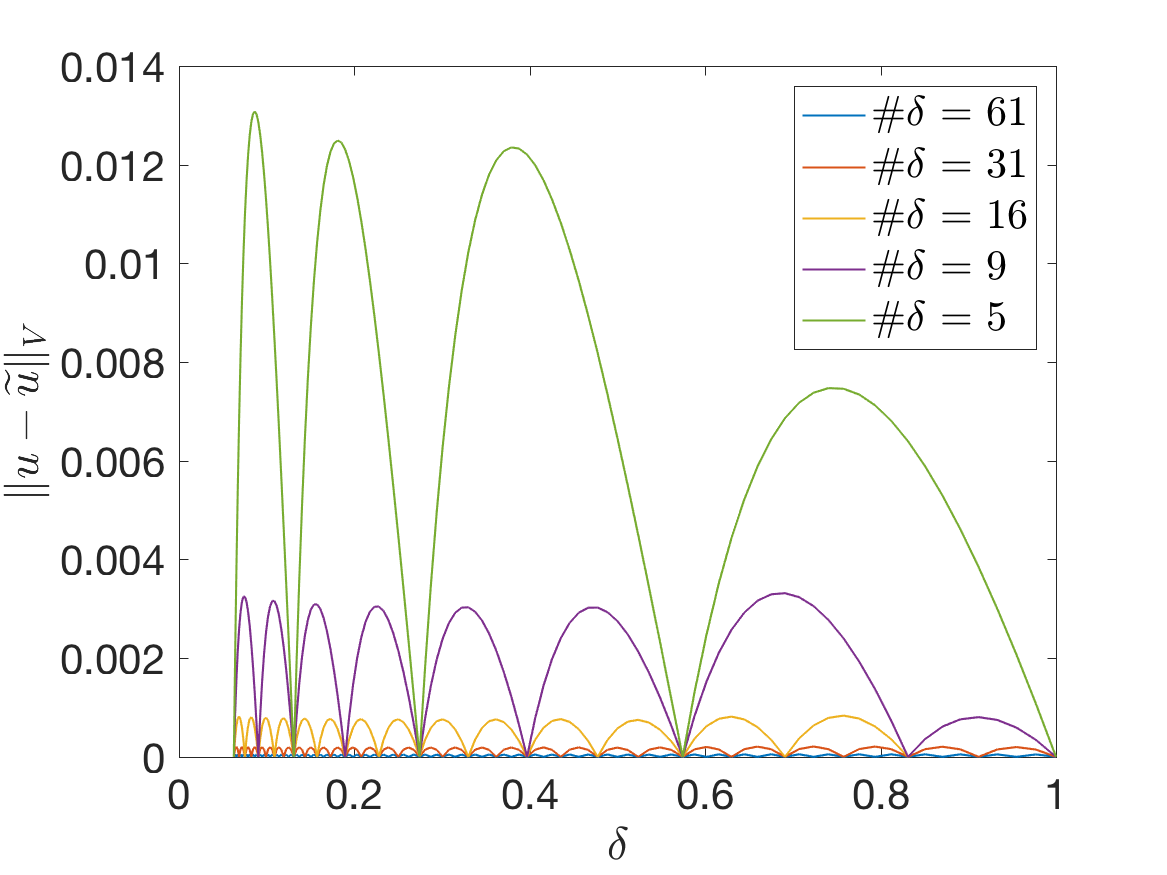}
\caption{Pointwise error $\norm{u(\delta)-\widetilde{u}(\delta)}_\V$ over the interval $\delta \in \P^\delta=[\deltamin,\deltamax]$ for \emph{Case~1} (left) and \emph{Case~2} (middle) and \emph{Case~2} using an graded refinement~\eqref{delta_graded} (right); $s=0.5$.}
\label{fig:1}
\end{figure}

The error convergence in terms of $\Delta\delta$ for \emph{Case~1} and \emph{Case~2} using uniform and adaptive refinements for $\kerneld$ in~\eqref{kernel_approx_gen} is presented on Figure~\ref{fig:2}. We observe linear for \emph{Case~1} and quadratic for \emph{Case~2} order of convergence for both type of refinements, which is in agreement with the theoretical results from Proposition~\ref{prop:error_sol_affine_delta}. We remark that for the chosen parameter $s=0.5$, the theory guarantees only $u \in \HD{1-\varepsilon}$ for arbitrarily small $\varepsilon > 0$, however, since the computation is based on the finite element solution, it still holds $u \in H^1_0(\Omega)$ on the discrete level, and the corresponding norm can depend only very weakly on the mesh size $h$.
\begin{figure}[htbp]
 \centering
\includegraphics[width=0.45\textwidth]{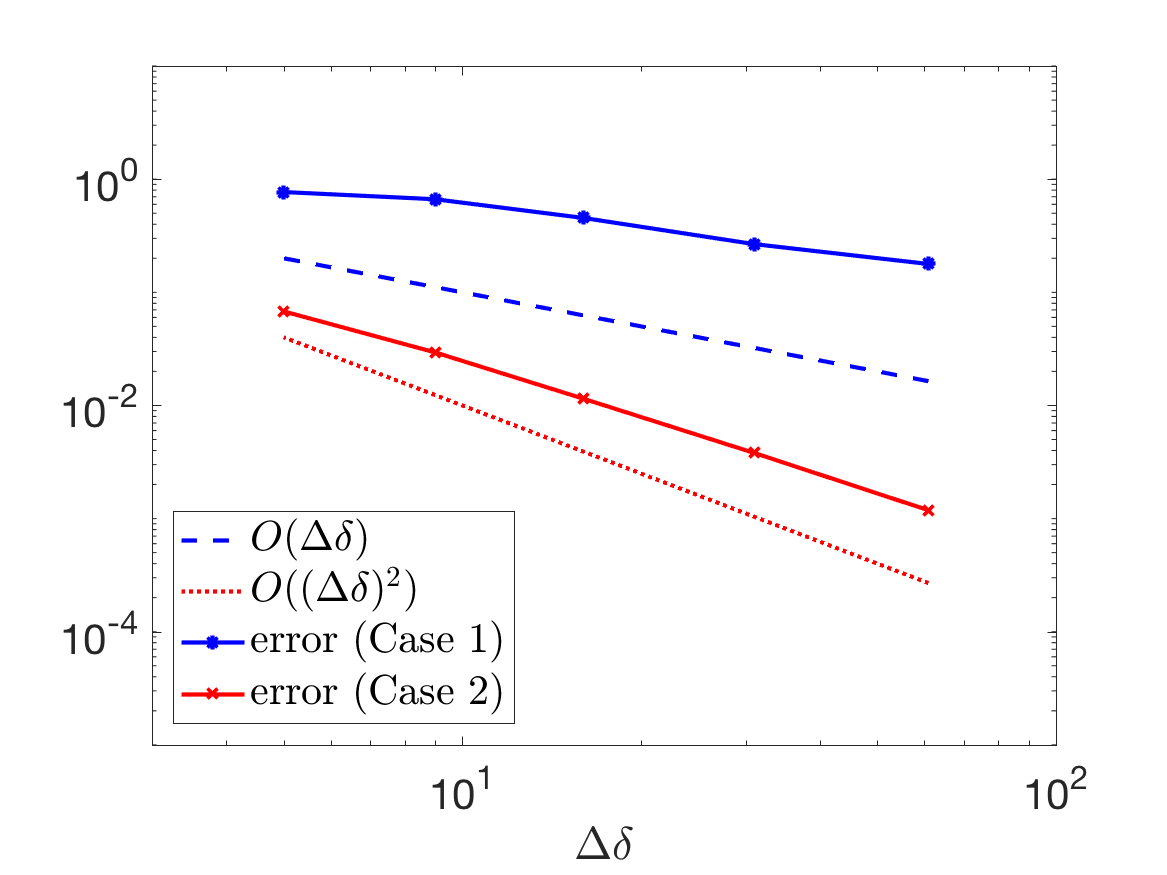}
\includegraphics[width=0.45\textwidth]{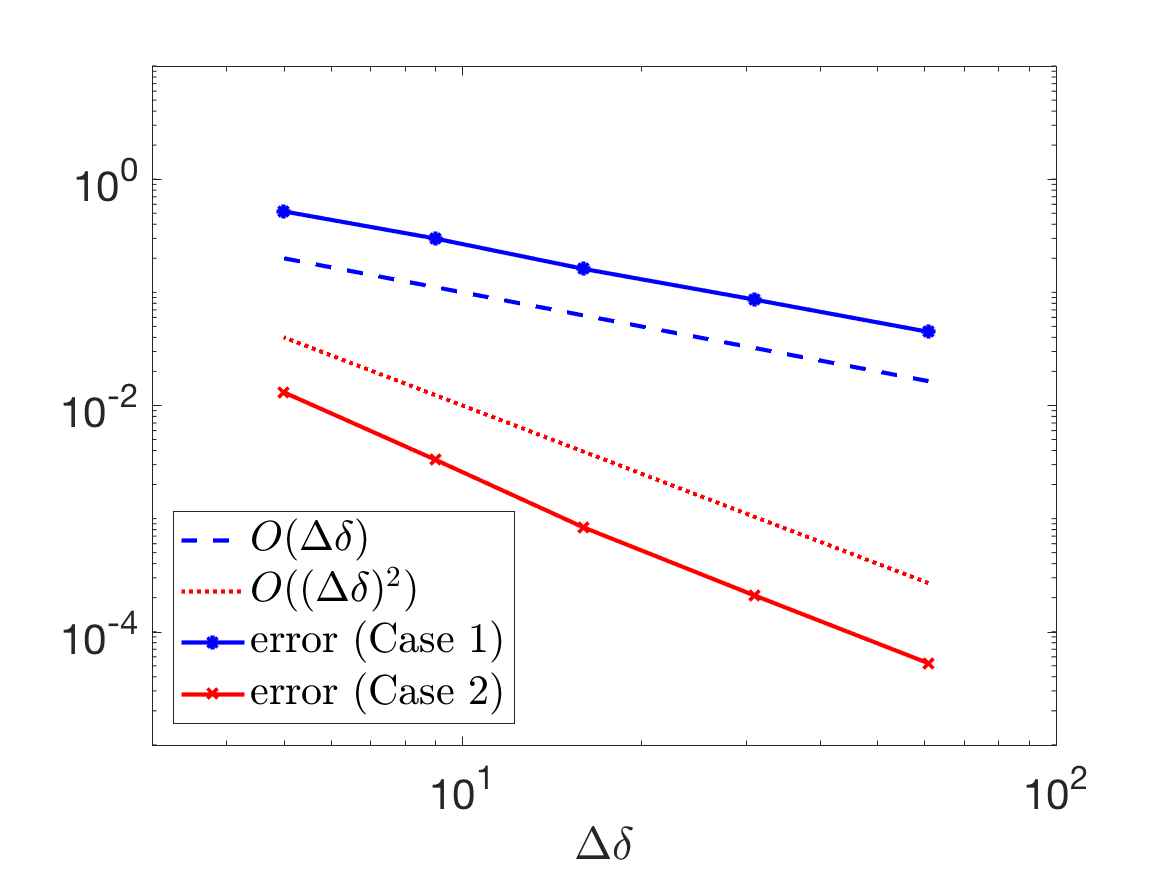}
\caption{Convergence of $\max_{\delta\in\P_{\rm train}^{\delta}}\norm{u(\delta)-\widetilde{u}(\delta)}_{\V}$ with respect to the number of points $K$ using uniform refinement (left) and graded refinement~\eqref{delta_graded} (right); $s=0.5$. }
\label{fig:2}
\end{figure}

In a similar manner, for the parameter $s$ we investigate the error 
between the discrete solution $u(s)$ of the original problem~\eqref{var_form_s} and the solution $u^{\rho}(s)$ of the regularized problem~\eqref{detailed_reg} with the regularization parameter $\rho = 2 C(\delta)\sigma^{M+1}$.
Here, we implement the regularization term with parameter $\s = \smin + 1/4 = 7/12$, which corresponds to a choice of $\varepsilon$ below the machine precision.
We note that this leads to $\sigma = (\smax - \smin)/(2\heps(\smax)) = 1/2$.
In Figure~\ref{fig:3}, the error is plotted for different number of Chebyshev points $M$ used in~\eqref{bilform_sapprox}. We clearly observe the exponential convergence of the error, which numerically confirms the error bound~\eqref{error_reg_detailed}. 
\begin{figure}[htbp]
 \centering
\includegraphics[width=0.5\textwidth]{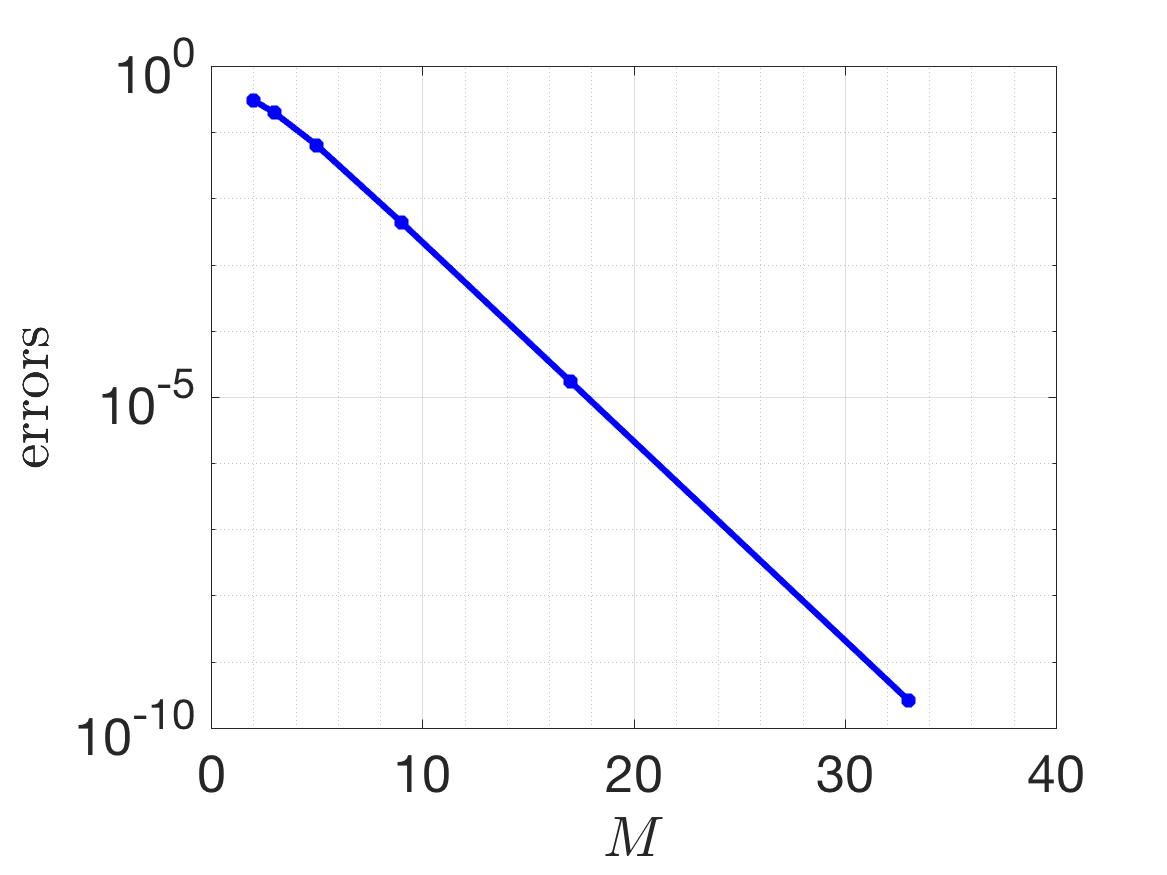}
\caption{Convergence of $\max_{s\in\P_{\rm test}^{s}}\norm{u(s)-{u}^\rho(s)}_{\VD{s}}$ for $\delta=0.25$.}
\label{fig:3}
\end{figure}

\subsection{Reduced basis approximation}
Finally, we investigate numerically the reduced basis approximation for $\delta$ and $s$. Using the greedy algorithm, based on the true error criterion, we iteratively construct the reduced basis spaces. In Figure~\ref{fig:4} we plot the convergence of the RB error for $\delta$ over the test set $\P^{\delta}_{\rm test}$ for different dimensions of the RB space $N$ and numbers of interpolation points $K$. That is, the error $\max_{\delta\in\P_{\rm test}^{\delta}}\norm{u(\delta)-{{u}}_N(\delta)}_{\VD{s}}$, where $u(\delta)$ and ${u}_N(\delta)$ are the discrete solutions of~\eqref{nonl_linear_var} and~\eqref{eq:rb_problem_delta}, respectively. Here, we set $s=0.5$.

Similarly, we consider the problem with the parameter $s$ for different values of $N$ and $M$. Concerning the choice of $\s$ and $\rho$, we follow the same settings as in the previous section. Here, we plot the relative error $\max_{s\in\P_{\rm test}^{s}}\norm{u(s)-{u}^\rho_N(s)}_{\VD{s}}/\norm{u_N^{\rho}(s)}_{\VD{s}}$.
\begin{figure}[htbp]
 \centering
\includegraphics[width=0.45\textwidth]{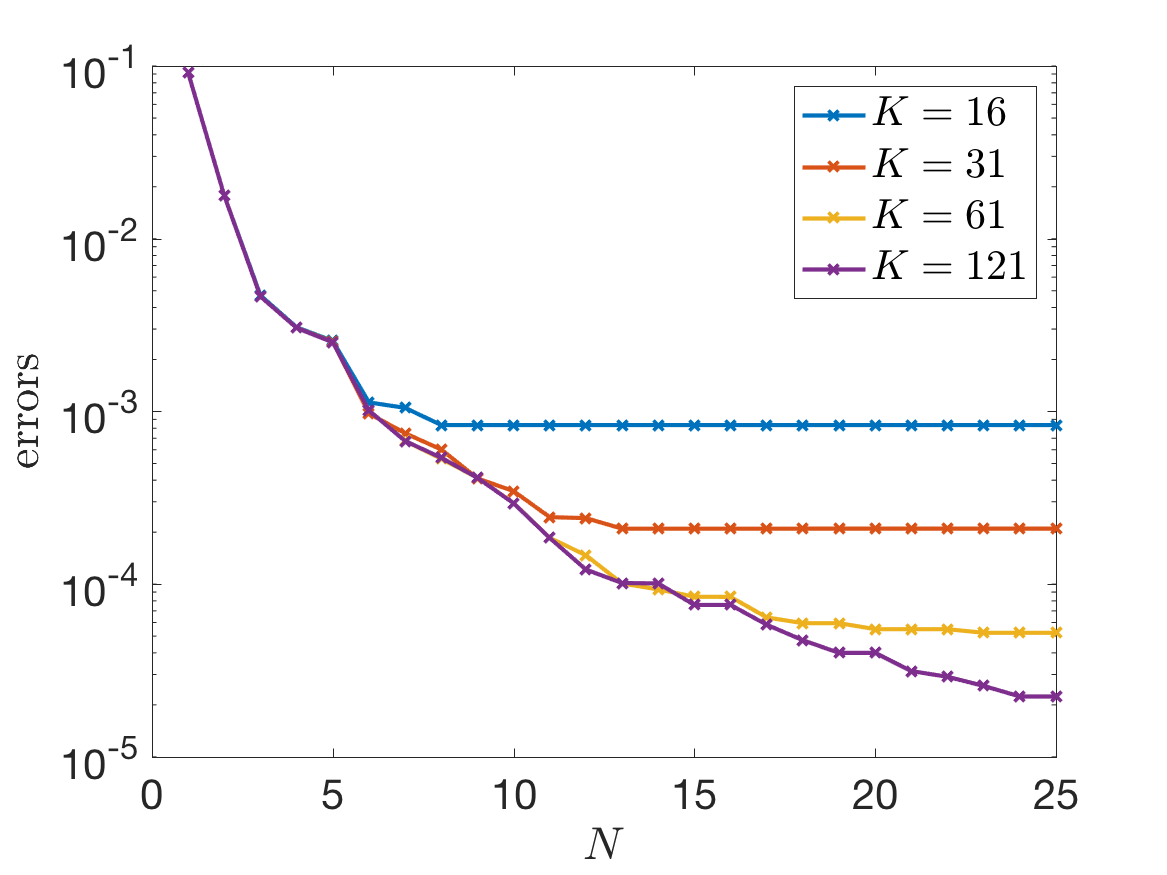}
\includegraphics[width=0.45\textwidth]{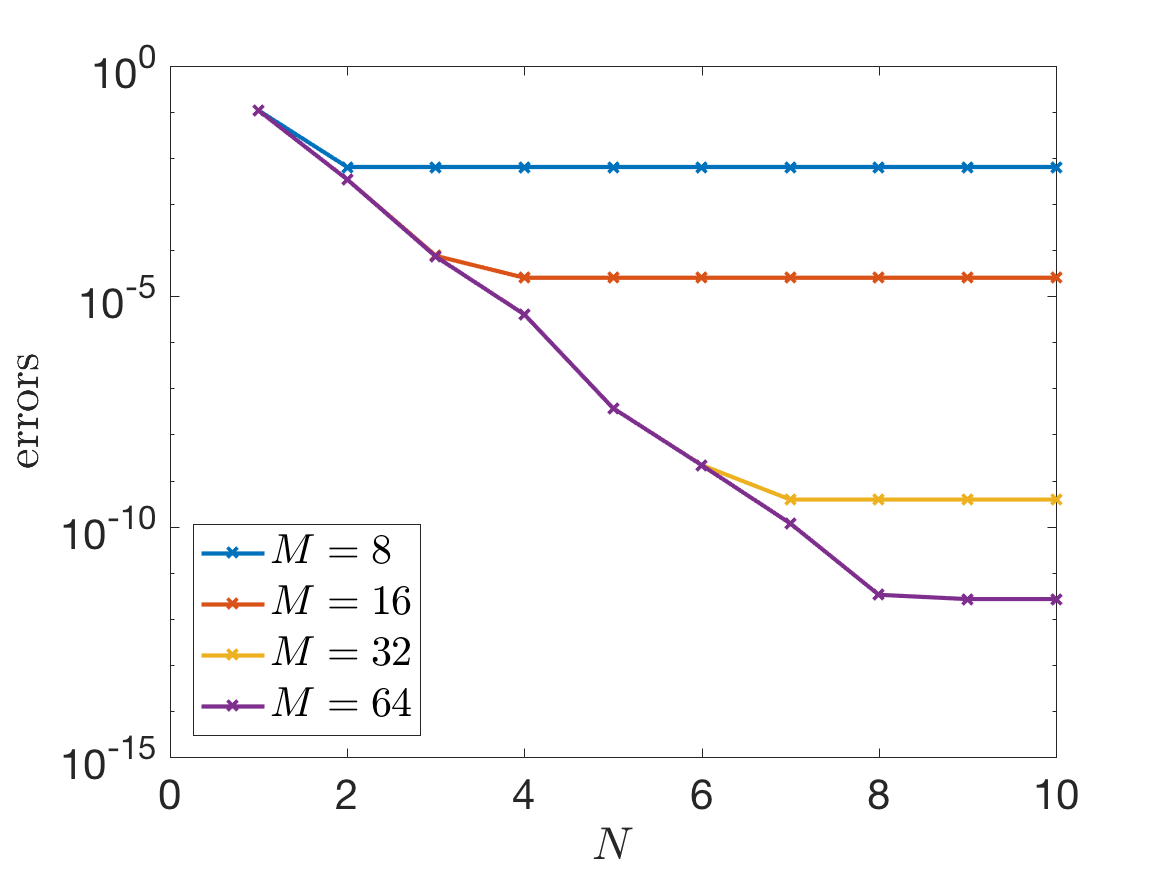}
\caption{Convergence of the reduced basis approximation for $\delta$ (left) and for $s$ (right). }
\label{fig:4}
\end{figure}
From both plots we observe a rapid convergence of the RB error. However, while we can clearly see the exponentially decaying behavior of the error in the parameter $s$, for the case of $\delta$ the convergence is slower, that could be explained by the reduced parameter regularity of the solution. In the case that the number of points $K$ (resp.\ $M$) is too low, the error can not be reduced below the error caused by the affine approximation. In the case of $s$, machine precision is reached already for $M=64$.



\newpage
\appendix

\section{Auxilliary estimates}
\label{app:aux}
\begin{proposition}\label{appx: prop1}
For $\delta\in(0,\infty)$, $k=1,2,\dots$ and $\alpha>0$, we have the following estimate
\begin{equation}
\label{estimate_log_k}
\sup_{\xi\in[0,\delta]}{\xi^{\alpha}}{|\log^k(\xi)}|
\leq \left(\frac{k}{e\alpha}\right)^k + \delta^{\alpha}(\log(\delta))_+^k\leq k! \left(\frac{1}{e\alpha^{k}} + \delta^{\alpha+1}\right).\nonumber
\end{equation}
\end{proposition}
\begin{proof}
For $\delta\in(0,1)$, $k=1,2,\dots$, let $f(\xi):={\xi^{\alpha}}{|\log^k(\xi)|}$. Then, we can express
\begin{equation}
\sup_{\xi\in[0,\delta]}f(\xi)\leq\sup_{\xi\in[0,1]}f(\xi)+\sup_{\xi\in[1,\delta]}f(\xi).\label{appx:proof1}
\end{equation}
We analyze each term on the right-hand side of~\eqref{appx:proof1} separately. For the first term with $\xi \leq 1$, we obtain $f(\xi) = (-1)^k \xi^\alpha \log^k(\xi)$. Clearly, $\lim_{\xi\to 0}f(\xi) = f(1) = 0$ and to find a maximum, we compute the first derivative as
\begin{equation*}
f^{\prime}(\xi)= (-1)^k \left[\alpha\xi^{\alpha-1}\log^k(\xi) + \xi^\alpha\frac{k\log^{k-1}(\xi)}{\xi}\right]=\xi^{\alpha-1}|\log^k(\xi)|\left(\alpha+\frac{k}{\log(\xi)}\right).
\end{equation*}
Then, the equation $f^{\prime}(\xi)=0$ for $\xi\in(0,1)$ is solved exactly for $\xi^{\star}=e^{-k/\alpha}$.
From this, we obtain
\[\max_{\xi\in[0,1]}f(\xi)=f(\xi^{\star})=e^{-k}\norms{\log^k\left(e^{-k/\alpha}\right)}=\left(\frac{k}{e\alpha}\right)^k.\]
For the second term, we only have to consider the case $\delta > 1$ (otherwise the term is zero and \(\log(\delta)_+ = 0\)):
Since $f(\xi)$ is monotonically increasing function on the interval $[1,\delta]$, the last term in~\eqref{appx: prop1} is realized at $f(\delta)$, i.e., $\sup_{\xi\in[1,\delta]}f(\xi)=\delta^{\alpha}\log^k(\delta)$. 
This proves the first inequality.

Finally, from the fact that $k!=\Gamma(k+1)$, $k\in\mathbb{N}$, we can write
\begin{equation*}
k!=\Gamma(k+1)=\int_{0}^{\infty}t^{k}e^{-t}\d t\geq\int_{c}^{\infty}t^{k}e^{-t}\d t\geq c^{k}\int_{c}^{\infty}e^{-t}\d t=c^{k}e^{-c},
\end{equation*}
where $c$ is some positive constant. Then, taking $c=\log(\delta)$, $\delta>1$, we obtain that $k!\geq\log^k(\delta)\delta^{-1}$.
Finally, using the fact that $(k/e)^ke \leq k!$, we obtain the second inequality.
\end{proof}
\end{document}